\documentclass[openany, 11pt]{article}
\usepackage[utf8]{inputenc}
\usepackage{amsmath,amscd,amssymb,amsfonts,amsthm,courier,relsize,bm,yhmath}
\usepackage{hyperref,enumerate,mathrsfs,mathtools,slashed}
\usepackage{blindtext}
\usepackage{titlesec}
\usepackage{tikz-cd}
\tikzcdset{ampersand replacement=\&}
\usepackage{xcolor}  %the next few lines set the colors of the hyperlinks, and remove the color boxes
\hypersetup{
    colorlinks,
    linkcolor={red!50!black},
    citecolor={blue!70!black},
    urlcolor={blue!80!black}
}

\newtheorem*{theorem*}{Theorem}
\newtheorem{theorem}{Theorem}[section]

\newtheorem{proposition}[theorem]{Proposition}

\theoremstyle{definition}    

\theoremstyle{remark}

\newtheorem*{example*}{\textbf{Example}}

\newcommand{\ignore}[1]{}

\def\la{\ensuremath{\langle}}
\def\ra{\ensuremath{\rangle}}
\def\lp{\ensuremath{\left(}}
\def\rp{\ensuremath{\right)}}
\def\F{\ensuremath{\mathcal{F}}}

\def\H{\ensuremath{\mathcal{H}}}

\def\Sc{\ensuremath{\mathscr{S}}}

\def\cal{\ensuremath{\boldsymbol{\circlearrowleft}}}

\def\bC{\ensuremath{\mathbb{C}}}
\def\bH{\ensuremath{\mathbb{H}}}
\def\bR{\ensuremath{\mathbb{R}}}
\def\bZ{\ensuremath{\mathbb{Z}}}

\def\bK{\ensuremath{\mathbb{K}}}

\def\bS{\ensuremath{\mathbb{S}}}
\def\bO{\ensuremath{\mathbb{O}}}

\def\End{\ensuremath{\textnormal{End}}}
\def\Hom{\ensuremath{\textnormal{Hom}}}

\def\dim{\ensuremath{\textnormal{dim}}}

\title{Explicit Families of Spinor Representations}
\author{Jesus Sanchez Jr}
\date{
    Mathematics Department, Washington University in St Louis\\[2ex]
    \today
}

\begin{document}
\maketitle
\begin{abstract}
    We provide a recipe for building explicit representations of the real Clifford algebras once an explicit family is given in dimensions $1$ through $4$. We further give an explicit construction of spin coordinate systems for a given real spinor module and use it to explicitly compute the parallel transport of spinor fields. We further highlight some novelties such as the relationship with the spectrum of the spinor Dirac operator and the Hodge de Rham operator when a parallel spinor field exists and a brief discussion of spinors along a hypersurface in $\bR^4$. Lastly, we extend our construction to arbitrary signature quadratic forms thus providing a complete and explicit family of spinor representations for all mixed signature Clifford algberas. We show that in all cases the spinor representations can be expressed as tensor products of multi-vectors over the fields $\bR$, $\bC$, and $\bH$.
\end{abstract}
{
  \hypersetup{linkcolor=black}
  \tableofcontents
}
\section{Introduction}
Given a real quadratic vector space $(V,g)$, the representation theory of both the real and complexified Clifford algebra is well understood and completely classified by realizing the corresponding Clifford algebra as a matrix algebra or direct sum of two matrix algebras. The definition of spinor module $S$ then becomes that of an irreducible module over the Clifford algebra, real or complex depending on the choice of real or complexified Clifford algebra. In the case of the complexified Clifford algebra $\bC\ell(V,g)$, there is a natural choice of spinor representation where one takes $S:=\wedge_{\bC}(V,J)$ for a chosen complex structure $J$ on $V$ when $V$ is even dimensional and $S:=\wedge_{\bC}(V\oplus\bR,J)$ when $V$ is odd dimensional. This representation can be used to show that any complex manifold is $Spin^c$ and is important in the discussion of the \emph{Hirzebruch-Riemann-Roch} theorem in complex geometry. If one restricts to the case of Euclidean Clifford algebras then there is a satisfactory answer which involves the use of octonions as in \cite{E} and \cite{RB}.  However, the real side suffers in that there are is no know unified treatment for the families of irreducible representations of $C\ell(V,g)$ which are explicit and geometric as in the case of the complexified Clifford algebras \emph{and} valid over all possible signatures. \par
There has been recent interest in understanding the geometry of spinors as well as generalizing the concept to new domains. The work of \cite{DER} presents a new take by defining the notion of $d$-blades with a view towards understanding the role of spinors in quantum field theory, whereas the work of \cite{L} presents an approach to spinor fields on the infinite dimensional loop space of a Riemmanian manifold $(M,g)$. The interest in understanding spinors derives from the application of the \emph{Atiyah-Singer index theorem} in spin geometry, which states that on a closed, oriented, spin Riemannian manifold $(M,g)$, the $\hat{A}$-genus is related to the \emph{spinor Dirac operator $\slashed{D}$} via 
\[
    \int_M\widehat{\textbf{\emph{A}}}(M)=\textbf{\emph{dim}}\text{Ker}\slashed{D}^+-\textbf{\emph{dim}}\text{Ker}\slashed{D}^-.
\]
As the above expression is only valid on spin Riemannian manifolds, having an explicit and geometric construction of the spinor bundle as well as the corresponding spin frame bundle which double covers the oriented orthonormal frame bundle of $(M,g)$ quickly becomes highly desirable.\par
\subsection*{Organization}
In this paper we will provide a recipe for cooking up explicit families of real spinor representations for the real Euclidean Clifford algebras, which only needs explicit representations in dimensions $1$-$4$. The key idea will be to use mixed field tensor products where the tensor products are allowed to vary from $\bR$, $\bC$, and $\bH$. From here we construct an explicit family of Euclidean spinors modeled on quaternionic multi-vectors as a well a family based on square-roots of space. Once a choice of spinor module is given, we discuss the notion of spin frames from this perspective as needed for passage to the spin vector bundle setting. For completeness, we include a section on the construction of spinor bundles which is rather standard but does also mention the spin orientation $\bR$-gerbe from our angle. We further include novel low dimensional phenomena including a discussion of the spin Levi-Civita connection of an oriented surface $\Sigma$, the treatment of spinors along hypersurfaces embedded in $\bR^4$, and a discussion of products of spinors giving a local Stoke's theorem for the Dirac operator. We conclude the paper with the generalization of our spinor modules to the setting of pseudo-Euclidean vector spaces, thus giving a complete collection of explicit geometric irreducible representations of real (pseudo)-Euclidean Clifford algebras.

\section{Preliminaries}
We begin by collecting a few standard definitions and structural results concerning Clifford algebras. The standard references for these can be found in \cite{ABS} or \cite{LM}. Given a finite dimensional real Euclidean vector space $(V,g)$, the \textbf{\emph{Clifford Algebra}} $C\ell(V,g)$ of $(V,g)$ is the universal unital $\bR$-algebra generated by vectors $v\in V$ subject to the relations
\[
    v\cdot w+w\cdot v=-2g(v,w),\quad \text{for all } v,w\in V.
\]
Using the $\bR$-tensor algebra of $V$
\[
    T_{\bR}(V):=\bigoplus_{j=0}^{\infty}V^{\otimes^j_{\bR}}
\]
an explicit construction of $C\ell(V,g)$ can be given as follows: take the two sided ideal of the tensor algebra $T_{\bR}(V)$ of $V$
\[
    I_g:=\la \bigl\{v\otimes w+w\otimes v+2g(v,w)\big|v,w\in V\bigr\}\ra
\]
and define
\[
    C\ell(V,g):=T_{\bR}(V)/I_g.
\]
A key property of $C\ell(V,g)$ can be obtained from this explicit construction. The \textbf{\emph{transpose automorphism}} on the tensor algebra is defined as
\begin{align*}
    \lp\cdot\rp^t:T_{\bR}(V)&\longmapsto T_{\bR}(V)^{op}\\
    v_1\otimes_{\bR}v_2\cdots\otimes_{\bR}v_k&\longmapsto v_k\otimes_{\bR}\cdots v_2\otimes_{\bR}v_1
\end{align*}
and preserves the ideal $I_g$ and thus descends to an algebra automorphism
\[
    \lp\cdot\rp^t:C\ell(V,g)\xrightarrow{\sim}C\ell(V,g)^{op}.
\]
Given any unital $\bR$-algebra $\mathcal{A}$ and $\bR$-linear map $\varphi:V\rightarrow \mathcal{A}$, if $\varphi$ satisfies the \textbf{\emph{Clifford condition}}
\[
    \varphi(v)\varphi(w)+\varphi(w)\varphi(v)=-2g(v,w),\quad\text{for all }v,w\in V
\]
then $\varphi$ lifts to a unique unital $\bR$-algebra morphism
\[
    \varphi:C\ell(V,g)\rightarrow \mathcal{A}.
\]
Note that if $(V,g)\simeq(\bR^n,g_n)$ is isometric to the standard Euclidean vector space of dimension $n$ then
\[
    C\ell(V,g)\simeq C\ell(\bR^n,g_n),
\]
and we will denote $C\ell(\bR^n,g_n)=C\ell_n$. What this effectively means is that we are free to work with one particular family of Clifford algebras when discussing algebraic properties of Clifford algebras and the structure of modules.\par
One of the key features that the real Clifford algebras $C\ell_n$ possess is the $8$-fold algebraic Bott periodicity which states that
\[
    C\ell_{n+8}\simeq C\ell_n\otimes_{\bR} C\ell_8, 
\]
where the $\otimes_{\bR}$ is the tensor product of $\bR$-algebras. This isomorphism effectively means that to determine the algebraic structure of $C\ell_n$ one need only determine the algebraic structure of $C\ell_1$ through $C\ell_8$. This is very well known and can be done explicitly as seen in \cite{ABS} or \cite{LM}
\begin{align*}
    C\ell_1&\simeq\bC,\quad
    C\ell_2\simeq\bH,\quad
    C\ell_3\simeq\bH\oplus\bH,\quad
    C\ell_4\simeq M_2(\bH)\\
    C\ell_5\simeq M_4(\bC)&,\quad
    C\ell_6\simeq M_8(\bR),\quad
    C\ell_7\simeq M_8(\bR)\oplus M_8(\bR),\quad
    C\ell_8\simeq M_{16}(\bR).
\end{align*}
In particular this shows that the Clifford algebras are either simple $\bR$-algebras or direct sums of simple $\bR$-algebras. This is particularly convenient for determining the structure of their modules. When the Clifford algebra $C\ell_n$ is isomorphic to a real matrix algebra $M_j(\bK)$ for $\bK$ a division algebra over $\bR$, as is the case when $n\neq 3\text{ (mod $4$)}$, there is up to isomorphism one unique irreducible module given as
\[
    S=\bK^j
\]
and the action is by the usual left action of matrices on vectors. When the Clifford algebra $C\ell_n$ is isomorphic to two copies of the same matrix algebra $M_j(\bK)\oplus M_j(\bK)$, as is the case when $n=3\text{ (mod $4$)}$, then there are two distinct irreducible modules up to isomorphism given by
\[
    S=\bK^j
\]
where either the first factor acts by matrix-vector multiplication and the second factor acts trivially or the first factor acts trivially and the second factor acts by matrix vector multiplication.\par
The real Clifford algebras come equipped with a real $\bZ_2$-grading automorphism $\varphi:C\ell(V,g)\rightarrow C\ell(V,g)$ defined on the generators by
\[
    \varphi(v)=-v\quad\text{for all }v\in V.
\]
The even and odd components of the Clifford algebras are defined to be the $\pm1$-eigenspaces of $\varphi$
\[
    C\ell^0(V,g)=\bigl\{ \omega\in C\ell_n\big|\quad\varphi(\omega)= \omega\bigr\},\quad C\ell^1(V,g)=\bigl\{ \omega\in C\ell_n\big|\quad\varphi(\omega)=- \omega\bigr\}
\]
Furthermore, in dimensions $n=0(\text{ mod } 4)$, the \textbf{\emph{real volume element}} $\boldsymbol{\nu}=e_1\cdots e_n\in C\ell_n$ squares to $1$
\[
    \boldsymbol{\nu}\cdot\boldsymbol{\nu}=1.
\]
This implies that if we are given a module $C\ell_n\circlearrowright M$, then this module becomes $\bZ_2$-graded by defining
\[
    M^{\pm}:=\lp 1\pm\boldsymbol{\nu}\rp\cdot M.
\]
In the case when $M$ is irreducible, this $\bZ_2$-grading is compatible with the right action of $\bK_n$. If we are given two a right $\bZ_2$-graded $\bK$-module $M$ and a left $\bZ_2$-graded $\bK$-module $N$, then we can form the tensor product becomes $M\hat{\otimes}_{\bK}N$ becomes $\bZ_2$-graded by defining
\begin{align*}
    \lp M\hat{\otimes}_{\bK}N\rp^+&=M^+\otimes_{\bK}N^+\oplus M^-\otimes_{\bK}N^-\\
    \lp M\hat{\otimes}_{\bK}N\rp^-&=M^+\otimes_{\bK}N^-\oplus M^-\otimes_{\bK}N^+
\end{align*}
The $\bZ_2$-grading on $M$ induces a $\bZ_2$-grading on $\End^R_{\bK}M$ by defining $T$ to be even if it preserves the splitting $M=M^+\oplus M^-$ and odd if it swaps the splitting. 
If $T\in\End^R_{\bK}M$ and $S\in\End^L_{\bK}N$ have a well defined grading and $m\otimes n\in M\hat{\otimes}N$ with $m\in M^{\pm}$, one defines
\[
    T\hat{\otimes}_{\bK}S\lp m\otimes_{\bK} n\rp:=(-1)^{\text{deg}S\cdot\text{deg}m}Tm\otimes_{\bK} Sn.
\]
For general $T\in\End^R_{\bK}M$ and $S\in\End^L_{\bK}N$ one extends the above definition bi-linearly on the graded components of $T$ and $S$.
\par
The groups $Pin(V,g)$ and $Spin(V,g)$ are defined as
\[
    Pin(V,g):=\la \bigl\{ v\in V\big|g(v,v)=1\bigr\}\ra\quad Spin(n):=Pin(V,g)\cap C\ell^0(V,g).
\]
If $C\ell^{\times}(V,g)$ denotes the collection of invertible elements of $C\ell(V,g)$, then $Pin(V,g)$ is a subgroup of $C\ell^{\times}(V,g)$ and thus inherits the \textbf{\emph{twisted adjoint representation}}
\[
    \widetilde{Ad}:C\ell^{\times}(V,g)\rightarrow GL(C\ell(V,g))
\]
given by
\[
    \widetilde{Ad}(\omega)(\tau):=\omega\cdot\tau\cdot\varphi(\omega)^{-1}.
\]
Note that if $\tau=v\in V$ and $\omega=w\in V$, then 
\[
    \widetilde{Ad}(w)(v)=v-2\frac{g(v,w)}{|w|^2}w,
\]
i.e. $\widetilde{Ad}(w)$ is the reflection through the plane perpendicular to $w\in V$. Thus this induces a $2$-to$1$ representation
\[
    \widetilde{Ad}:Pin(V,g)\rightarrow O(V,g)
\]
which when restricted to $Spin(V,g)$ gives a double cover of $SO(V,g)$
\[
    \widetilde{Ad}:Spin(V,g)\rightarrow SO(V,g).
\]
If we are given an irreducible left real module $S$ for $C\ell(V,g)$ then the \textbf{\emph{$S$-spin representation}} is defined to be the restriction of
\[
    \pi_S:C\ell(V,g)\rightarrow \End_{\bR}(S)
\]
to $Spin(V,g)$
\[
    \pi_S:Spin(V,g)\rightarrow GL(S,\bR).
\]
For a given irreducible left module $S$ of $C\ell(V,g)$ the \textbf{\emph{intertwiner algebra}}
\[
    \End_{C\ell}(S):=\bigl\{T\in\End_{\bR}(S)\big|\quad T(c(\omega)s)=c(\omega)T(s)\text{ for all }\omega\in C\ell(V,g),\text{ and }s\in S\bigr\}
\]
is isomorphic to a $\bR$-division algebra. Moreover, since all $\bR$-division algebras are isomorphic to their opposite algebra, we have that there is an isomorphism
\[
    \chi:\End_{C\ell}(S)\xrightarrow{\sim}\End_{C\ell}(S)^{op}.
\]
Thus, we can consider the left $C\ell(V,g)$-module as a $C\ell(V,g)$-$\End_{C\ell}(S)$-bimodule
\[
    C\ell(V,g)\circlearrowright S\circlearrowleft \End_{C\ell}(S).
\]
Note that in the case of $C\ell_n$ there is an isomorphism of algebras $\psi:C\ell_{n-1}\xrightarrow{\sim} C\ell^0_n$ induced by
\[
    \psi(v):=v\cdot e_n,\quad\text{for all }v\in\bR^{n-1}.
\]
This has an interesting consequence when discussing the intertwiners of irreducible modules of $C\ell_n$. For example, consider $C\ell_5$ and its irreducible action on $\bR^8$ which has an algebra of intertwiners isomorphic to $\bC$. If we restrict this action to $C\ell^0_5\simeq C\ell_4$, then the algebra of intertwiners grows to an algebra isomorophic to $\bH$. What this entails is that if $S$ is an irreducible $C\ell(V,g)$-module with intertwiner algebra isomorphic to $\bK_n$, when one restricts this representation to $Spin(V,g)$, the intertwiner algebra may grow to a larger algebra of intertwiners. For this reason we denote by $\bK^0_n$ the intertwiner algebra for the the restriction of an irreducible action of $C\ell_n$ to $C\ell^0_n$. These algebras are written up to isomorphism below where $n$ is taken $(\text{mod } 8)$
\begin{align*}
    \bK_1^0\simeq& M_2\lp\bR\rp,\quad \bK^0_2\simeq M_2\lp\bC\rp,\quad \bK^0_3\simeq \bH\quad \bK^0_4\simeq \bH\\
    \bK^0_5&\simeq\bH,\quad\bK^0_6\simeq \bC,\quad\bK^0_7\simeq \bR,\quad\bK^0_8\simeq\bR
\end{align*}
whereas
\begin{align*}
    \bK_1\simeq\bC,\quad \bK_2\simeq\bH,\quad\bK_3\simeq \bH,\quad \bK_4\simeq\bH\\
    \bK_5\simeq \bC,\quad \bK_6\simeq\bR,\quad\bK_7\simeq\bR,\quad\bK_8\simeq\bR.
\end{align*}
Note that $\bK_n\subseteq\bK^0_n$ and the \emph{intersection} of the centers of $\bK_n$ and $\bK^0_n$ is always naturally isomorphic to $\bR$. This has the following consequence: If $S_n$ is an irreducible module of $C\ell_n$, then the real vector space of left-$C\ell_n$-right-$\bK^0_n$ linear maps $\End_{C\ell,\bK^0_n}\lp S_n\rp$ is one dimensional.  \par

The main idea we will use to construct explicit modules for the Clifford algebras goes as follows: By the above classification, if we can find a module for a given $C\ell_n$ which has the same dimension over $\bK_n$ as the irreducible modules for $C\ell_n$, then by uniqueness this must be an irreducible module for $C\ell_n$. The catch of course is to find these explicit representations!
\section{The General Recipe}
Suppose $S_i$ are irreducible real representations of $C\ell_i$ for $1\leq i\leq 4$. We denote by 
\[
    c_i:\bR^i\rightarrow \End_{\bR}(S_i)
\]
the $\bR$-linear map which induces the representation $C\ell_i\circlearrowright S_i$. Note that in the case when $i=3$, if we compose $c_3$ with $-1$, we obtain
\[
    -c_3:\bR^3\rightarrow\End_{\bR}\lp S\rp
\]
which one can check generates a distinct irreducible representation of $C\ell_3$ which we denote by $-S_3$. Using the fact that the intertwiner algebra is symmetric $\End_{C\ell_i}(S_i)\simeq\End_{C\ell_i}(S_i)^{op}$ and the Clifford algebras are also symmetric via the transpose automorphism, we can treat any of our irreducible modules as $\End_{C\ell}(S_i)-C\ell_i$ bimodules
\[
    \End_{C\ell}(S_i)\circlearrowright S_i\circlearrowleft C\ell_i.
\]
In effect, what we have done is show that there are two versions of a given irreducible representation $S_i$, one for which the Clifford action is left-$\bK_i$-linear or right-$\bK_i$-linear. This distinction is a superfluous when $\bK_i\simeq\bR$ or $\bK_i\simeq\bC$, but is necessary when $\bK_i\simeq\bH$. We will label the Clifford map by $L$ or $R$ to distinguish these actions
\[
    c^L_i:\bR^i\rightarrow\End^L_{\bK_i}(S_i),\quad c^R_i:\bR^i\rightarrow\End^R_{\bK_i}(S_i).
\]
We further remark that we already know that we can take $\End_{C\ell}(S_i)\simeq \bH$ for $2\leq i\leq 4$ and $\End_{C\ell}(S_1)\simeq\bC$.\par
We are now in a position to construct all irreducible representations for $C\ell_n$ for \emph{any $n$} using the known representations $S_i$, $1\leq i\leq 4$. We first construct the representation $S_8$. Given $\lp S_4,c^R_r\rp$ and $\lp S_4,c^L_4\rp$, and recalling the $\bZ_2$-grading by the volume element, we can take the $\bZ_2$-graded $\bH$-tensor product of these modules obtaining
\[
    S_8:=S_4\hat{\otimes}_{\bH}S_4.
\]
We now obtain a $\bR$-linear map 
\begin{align*}
    c_8:\bR^4\oplus\bR^4&\longrightarrow \End_{\bR}\lp S_4\hat{\otimes}_{\bH}S_4\rp\\
    \lp u,v\rp&\longmapsto c^R_4(u)\hat{\otimes}_{\bH}1+1\hat{\otimes}_{\bH}c^L_4(v)
\end{align*}
which satisfies the Clifford condition and thus lifts to a representation 
\[
    C\ell_8\circlearrowright S_4\hat{\otimes}_{\bH}S_4
\]
which is irreducible by dimension comparison. Furthermore note that this module is $\bZ_2$-graded over $\bR$ since $8=0(\text{ mod $4$})$. Moreover, note that we can cook up the corresponding irreducible modules for $C\ell_{8k}$ by simply taking $\bZ_2$-graded $\bR$-tensor products of the above module with itself
\[
    S_{8k}:=\lp S_8\rp^{\hat{\otimes}^k_{\bR}}
\]
and defining
\begin{align*}
    c_{8k}:\lp\bR^8\rp^{\oplus^k}&\longrightarrow \End_{\bR}\lp S_{8k}\rp\\
    \lp u_1,\cdots,u_k\rp&\longmapsto \sum_{j=1}^{k}I^{\hat{\otimes}^{j-1}_{\bR}}\hat{\otimes}_{\bR}c_8(u_j)\hat{\otimes}_{\bR}I^{\hat{\otimes}^{k-j}_{\bR}}
\end{align*}
Now consider the general case of $n=8k+r$ with $1\leq r\leq 7$. If $1\leq r\leq 4$ we take
\[
    S_{8k+r}:=\lp S_8\rp^{\hat{\otimes}^k_{\bR}}\otimes_{\bR}S_r.
\]
We make an important point about the graded tensor products seen above: Although $S_{8k}:=\lp S_8\rp^{\hat{\otimes}^k_{\bR}}$ is perfectly well defined, the tensor product with $S_r$ will be taken as graded when $r=4$ and ungraded otherwise. However, when taking the tensor product of $I_{S_{8k}}$ with $c^R_r$, we will still take the standard $\bZ_2$-graded tensor product of operators
\begin{align*}
    \lp I_{S_{8k}}\hat{\otimes}_{\bR}c^R_1\rp(x) = \left\{
        \begin{array}{ll}
             I_{S_{8k}}\otimes_{\bR}c^R_1(x)  & \quad \text{on } S^+_{8k}\otimes_{\bR}S_r, \\
            - I_{S_{8k}}\otimes_{\bR}c^R_1(x)  & \quad \text{on } S^-_{8k}\otimes_{\bR}S_r
        \end{array}
    \right.
\end{align*}
for $x\in\bR^r$. We now take
\begin{align*}
    c_{8k+r}:\lp\bR^8\rp^{\oplus^k}\oplus\bR^r&\longrightarrow\End^R_{\bK_r}\lp S_{8k+r}\rp\\
    \lp u,v\rp&\longmapsto c_{8k}(u)\otimes_{\bR}I_{S_r}+I_{S_{8k}}\hat{\otimes}_{\bR}c^R_r(v)
\end{align*}
and observe that this satisfies the Clifford condition so lifts to an irreducible representation of $C\ell_{8k+r}$ by dimension comparison. When $5\leq r\leq 7$, first note 
\[
    S_{8k+4}=\lp S_8\rp^{\hat{\otimes}^k_{\bR}}\hat{\otimes}_{\bR}S_4
\]
is $\bZ_2$-graded and note that the grading and $c_{8k+4}$ are compatible with the right $\bH$-action. We now take
\[
    S_{8k+r}:=S_{8k+4}\otimes_{\bK_{r-4}}S_{r-4}
\]
with
\begin{align*}
    c_{8k+r}:\lp\lp\bR^8\rp^{\oplus^k}\oplus\bR^4\rp\oplus\bR^{r-4}&\longrightarrow \End_{\bR}\lp S_{8k+r}\rp\\
    \lp u,v\rp&\longmapsto c^R_{8k+4}(u)\otimes_{\bK_{r-4}} I_{S_{r-4}}+I\hat{\otimes}_{\bK_{r-4}}c^L_{r-4}(v).
\end{align*}
The above map satisfies the Clifford condition and thus lifts to an irreducible representation of $C\ell_{8k+r}$ by dimension comparison. In the case when $r=3\text{ (mod $4$)}$ there is still one more irreducible representation floating around which can be accounted for by replacing $S_3$ by $-S_3$ in the above construction. We remark that in general, to construct an irreducible module for $C\ell_n$ out of the four $S_i$ one typically needs tensor products not only over $\bR$ but also over $\bC$ and $\bH$. This may explain the internal complexity of the representations and why it took some time for this construction to come to light. 
\subsection{The First Explicit Family}
We construct here a family of modules for $C\ell_n$ modeled on the multi-vectors of $\bH$. To begin with, note that when $n=1$, if we consider the inclusion
\begin{align*}
    c_1:\bR&\rightarrow \text{Im}\bC\\
    x&\mapsto ix
\end{align*}
then since 
\[
    c_1(x)c_1(y)+c_1(y)c_1(x)=-2x\cdot y,\quad \text{for all } x,y\in\bR,
\]
we see that $c_1$ lifts to a unital $\bR$-algebra morphism
\[
    c_1:C\ell_1\rightarrow \bC.
\]
One can check by dimension comparison that this is an isomorphism of $\bR$-algebras. In particular we can take as our spinor module in dimension $1$ to be $\bC$ with the left action coming from $c_1$.\par
In dimension $2$, if we consider the standard $\bR$-linear inclusion 
\[
    c_2:\bR^2\rightarrow \bR^3
\]
and then treat $\bR^3$ as the imaginary part of $\bH$, then we obtain an inclusion 
\[
    c_2:\bR^2\rightarrow \text{Im}\bH
\]
which satisfies 
\[
    c_2(x)c_2(y)+c_2(y)c_2(x)=-2g_2(x,y)\quad \text{for all }x,y\in\bR^2.
\]
Thus $c_2$ lifts to a unital $\bR$-algebra morphism
\[
    c_2:C\ell_2\rightarrow\bH
\]
which is an isomorphism by dimension comparison. In particular we can take our spinor module in dimension $2$ to be $\bH$ with the left action coming from $c_2$.\par
In dimension $3$, we can use the same observation in $n=2$ and identify
\[
    c_3:\bR^3\rightarrow \text{Im}\bH,
\]
in which case $c_3$ satisfies 
\[
    c_3(x)c_3(y)+c_3(y)c_3(x)=-2g_3(x,y),\quad \text{for all }x,y\in\bR^3.
\]
This will give us an induced unital $\bR$-algebra morphism
\[
    c_3:C\ell_3\rightarrow \bH.
\]
Note however that this will \textbf{\emph{not}} be an isomorphism since the $\bR$ dimension of $C\ell_3$ is $8$ while that of $\bH$ is 4. Note however that $c_3$ does produce a spinor module via
\begin{align*}
    C\ell_3\times \bH&\rightarrow \bH\\
    (x,q)&\mapsto c_3(x)\cdot q.
\end{align*}
This module is irreducible via the classification $C\ell_3\simeq \bH\oplus\bH$. To obtain the other spinor module, simply compose $c_3$ with multiplication by $-1$ on $\bR^3$, i.e.
\begin{align*}
   -c_3:\bR^3&\rightarrow\text{Im}\bH\\
    x&\mapsto -c_3(x).
\end{align*}
One has that
\[
    c_3(x)c_3(y)+c_3(y)c_3(x)=-2g_3(x,y)\quad\text{for all }x,y\in\bR^3
\]
and thus $-c_3$ lifts to a unital $\bR$-algebra morphism
\[
    -c_3:C\ell_3\rightarrow\bH.
\]
Thus we have obtained another irreducible spinor module for $C\ell_3$. To see that the modules given by $c_3$ and $-c_3$ are distinct, note that for the element $\omega=e_1e_2e_3\in C\ell_3$, for $e_1,e_2,e_3$ the standard basis vectors of $\bR^3$, one has
\[
    c_3(\omega)=-1,\quad -c_3(\omega)=1.
\]
Thus there can be no intertwining map between the modules determined by $c_3$ and $-c_3$. We take these to be our spinor modules in dimension $3$. Note that $c_i$ for $1\leq i\leq 3$ are naturally right $\bK_i$-linear.\par
%Note that in the above low dimensional examples we could have equally well consider the spinors as a \textbf{\emph{right module}} over $C\ell_n$. This will be relevant when constructing the spin representations in dimensions $5,6$ and $7$. To distinguish the left module representations from the right module represntations we will add the superscript $L$ or $R$ to $c_n$ and $c_3$ to denote left or right, respectively.\par
Things take an interesting turn in dimension $4$. Firstly, we no longer can employ the trick of embedding the vector space $\bR^4$ into the imaginary part of a larger \emph{associative} $\bR$-division algebra for which the Clifford relations will be satisfied. Rather here we will treat $\bR^4$ as $\bH$ and take as our spinor module $\wedge_{\bH}\bH$. We first briefly discuss the $\bH$-bimodule $\wedge_{\bH}\bH$. Note that if $\mathcal{A}$ is a unital $\bR$-algebra, and $M$ is a \textbf{\emph{bimodule}} over $\mathcal{A}$, then the $\mathcal{A}$-tensor algebra is
\[
    T_{\mathcal{A}}(\mathcal{M}):=\bigoplus_{k=0}^{\infty}T^k_{\mathcal{A}}(\mathcal{M})
\]
where for $k>0$
\[
    T^k_{\mathcal{A}}(\mathcal{M}):=\mathcal{M}\otimes_{\mathcal{A}}\mathcal{M}\otimes_{\mathcal{A}}\cdots\otimes_{\mathcal{A}}\mathcal{M}
\]
where there are $k$-tensor factors appearing in the above tensor product. For $k=0$ we take the convention
\[
    T^0_{\mathcal{A}}(\mathcal{M})=\mathcal{A}.
\]
Note that $T_{\mathcal{A}}(\mathcal{M})$ remains a $\mathcal{A}$-bimodule in a natural manner. Thus the $\mathcal{A}$-tensor algebra $T_{\mathcal{A}}(\mathcal{M})$ is an $\mathcal{A}$-$\mathcal{A}$-algebra. If we take $I_{\wedge}$ to be the two sided ideal
\[
    I_{\wedge}:=\la\bigl\{m\otimes_{\mathcal{A}}m\big|m\in\mathcal{M}\bigr\}\ra,
\]
then the $\mathcal{A}$-exterior algebra of $\mathcal{M}$ is defined to be 
\[
    \wedge_{\mathcal{A}}\mathcal{M}:=T_{\mathcal{A}}(\mathcal{M})/I_{\wedge}.
\]
Note that $\wedge_{\mathcal{A}}\mathcal{M}$ remains a $\mathcal{A}$-$\mathcal{A}$-algebra. We are interested in employing this construction when $\mathcal{A}=\mathcal{M}=\bH$. In this case we obtain the quaternic exterior algebra
\[
    \wedge_{\bH}\bH.
\]
Note that we have $\wedge_{\bH}\bH\simeq \bH\oplus\bH$ as $\bH$-bimodules. If we further note that $\bH$ has a right quaternic metric
\[
    g_{\bH}^R:\bH\times\bH\rightarrow\bH
\]
given via
\[
    g^R_{\bH}(q,w):=\overline{q}\cdot w,
\]
then for any $q\in\bH$ we obtain a left contraction operator $\iota^L_q:\wedge_{\bH}\bH\rightarrow\wedge_{\bH}\bH$. If we write $\wedge_{\bH}\bH=\wedge^0_{\bH}\bH\oplus\wedge^1_{\bH}\bH$ then in matrix form the operator $\iota^L_q$ is
\[
    \iota^L_q=\begin{bmatrix}
        0 & \overline{q}\cdot \\
        0 & 0 \\
    \end{bmatrix}
\]
where $\overline{q}\cdot$ denotes \textbf{\emph{left multiplication}} by $\overline{q}$. In particular $\iota^L_q$ is right-$\bH$-linear. We also have the left wedge operator defined $\varepsilon^L_q:\wedge_{\bH}\bH\rightarrow\wedge_{\bH}\bH$ for any $q\in\bH$ where in matrix form it is given as
\[
    \varepsilon^L_q=\begin{bmatrix}
        0 & 0 \\
        q\cdot & 0 \\
    \end{bmatrix}
\]
here again $q\cdot$ denotes left multiplication. Note also that $\varepsilon^L_q$ is right-$\bH$-linear for any $q\in\bH$. We now define
\begin{align*}
    c^R_4:\bH&\rightarrow \End^R_{\bH}(\wedge_{\bH}\bH)\\
    q&\mapsto \varepsilon^L_q-\iota^L_q.
\end{align*}
First note that $c^R_4(q)$ is \emph{only} $\bR$-linear in $q$. Secondly, note that 
\[
    c^R_4(q)c^R_4(w)+c^R_4(w)c^R_4(q)=-2g_4(q,w)\quad\text{for all }q,w\in\bH.
\]
Thus $c^R_4$ lifts to unital $\bR$-algebra morphism
\[
    c^R_4:C\ell_4\rightarrow\End^R_{\bH}(\wedge_{\bH}\bH).
\]
By dimension comparison this will be an isomorphism and we have thus obtained an explicit model for real spinors in dimension $4$ using quaternionic multi-vectors. Furthermore, in this case the spinor module is naturally $\bZ_2$-graded by taking the degree of the multivector to be the choice of grading.\par
By the general recipe given above, we see that for $n=8k+r$, $k\geq0$ we can take
\begin{align*}
    S_{8k+r}&:=\lp \wedge_{\bH}\bH\hat{\otimes}_{\bH}\wedge_{\bH}\bH\rp^{\hat{\otimes}^k_{\bR}}\otimes_{\bR}\bC,\quad r=1\\
    S_{8k+r}&:=\lp \wedge_{\bH}\bH\hat{\otimes}_{\bH}\wedge_{\bH}\bH\rp^{\hat{\otimes}^k_{\bR}}\otimes_{\bR}\bH,\quad r=2,3\\
    S_{8k+r}&:=\lp \wedge_{\bH}\bH\hat{\otimes}_{\bH}\wedge_{\bH}\bH\rp^{\hat{\otimes}^k_{\bR}}\hat{\otimes}_{\bH}\wedge_{\bH}\bH,\quad r=4\\
    S_{8k+r}&:=\lp \wedge_{\bH}\bH\hat{\otimes}_{\bH}\wedge_{\bH}\bH\rp^{\hat{\otimes}^k_{\bR}}\hat{\otimes}_{\bH}\wedge_{\bH}\bH\otimes_{\bC}\bC,\quad r=5\\
    S_{8k+r}&:=\lp \wedge_{\bH}\bH\hat{\otimes}_{\bH}\wedge_{\bH}\bH\rp^{\hat{\otimes}^k_{\bR}}\hat{\otimes}_{\bH}\wedge_{\bH}\bH\otimes_{\bH}\bH,\quad r=6,7\\
    S_{8k+r}&:=\lp \wedge_{\bH}\bH\hat{\otimes}_{\bH}\wedge_{\bH}\bH\rp^{\hat{\otimes}^{k+1}_{\bR}},\quad\quad r=8
\end{align*}
The astute reader may have noticed that tensoring by $\bC$ or $\bH$ in dimensions $8k+r$, $5\leq r\leq 7$ is redundant within these models for $S_{8k+r}$. One could simply take $S_{8k+r}$ to be some appropriate tensoring of $\wedge_{\bH}\bH$ with itself $k+1$-times, and then define the remaining action by the $r$-variables on the right by multiplication by $i$, $j$, or $k$. However, as we shall see below, there are different models for $S_r$ for which a canonical identification of $S_r$ with $\bC$, $\bH$, or $\wedge_{\bH}\bH$ is not available. 
\section{The Notion of Spin Coordinate System}
Recall that for $(V,g,\cal_{\bR})$, a real, oriented Euclidean vector space of $\bR$-dimension $n$, one can consider the set of oriented orthogonal frames $P_{SO}(V,g,\cal_{\bR})$ defined as the set of lists of vectors $\{ e_1,\cdots,e_n\}$ such that the elements of the list are unit length, pair-wise orthogonal, and the full list forms a basis of $V$ which is compatible with the orientation on $V$. An alternative perspective is to view elements of $P_{SO}(V,g,\cal_{\bR})$ as linear isometries 
\[
    \phi:(\bR^n,g_n)\rightarrow (V,g)
\]
which preserves the orientation (here $(\bR^n,g_n)$ is given the canonical orientation). From the latter perspective we can consider $P_{SO}(V,g,\cal_{\bR})$ as the ``moduli space of oriented orthonormal coordinate systems" for $(V,g,\cal_{\bR})$. The two interpretations of $P_{SO}(v,g,\cal_{\bR})$ are clearly equivalent and we will freely interchange between the two when convenient. For us, the latter perspective on frames will be best when describing the spin coordinate systems of $(V,g,\cal_{\bR})$.\par
Consider $(V,g,\cal_{\bR})$ along with an irreducible left module $\Sc_V$ for $C\ell(V,g)$. We further require that $\Sc_V$ admits the following extra \textbf{\emph{bells and whistles}}
\begin{itemize}
    \item[(I)] A right action of the division algebra $\bK_n$ for which the Clifford algebra is a matrix algebra over in such a way that $C\ell(V,g)\circlearrowright\Sc_V\circlearrowleft\bK_n$ becomes a $C\ell_n-\bK_n$-bimodule. This is effectively a choice of unital $\bR$-algebra isomorphism $\End_{Cl}\lp\Sc_V\rp\simeq \bK_n$. For $\bK_n=\bR$ or $\bK_n=\bC$ the choice is unique, but for $\bK_n=\bH$ there is a moduli of choices.
    \item[(II)] A \emph{real} inner product $g_{\Sc}$ such that for any given $v\in V$ and $a\in \text{Im}\bK_n$ the actions $c^L(v)$ and $a_R$ are skew-adjoint
    \item[(III)] A right action by the algebra $\bK^0_n$ for which $C\ell^0_n\circlearrowright\Sc_V\circlearrowleft\bK^0_n$ is a $C\ell^0_n-\bK^0_n$ bimodule and moreover for which the right actions by $\bK_{n}$ and $\bK^0_n$ are compatible with respect to the natural inclusion $\bK_n\subseteq\bK^0_n$. Once again, this is equivalent to a choice of untial $\bR$-algebra isomorphism $\End_{C\ell^0}\lp\Sc_V\rp\simeq \bK^0_n$ which extends the chosen isomorphism $\End_{C\ell}(\Sc_V)\simeq\bK_n$ in (I). In general, there may be a moduli of possible choices.
\end{itemize}
The existence of such structures on \emph{any} given module $\Sc_V$ is given in \cite{ABS} and \cite{LM}. A rather important remark with regards to such metrics $g_{\Sc}$ is that not only do they exist, but they are also determined up to a positive constant, i.e. if $g^1_{\Sc}$ and $g^2_{\Sc}$ both satisfy condition (II), then there exists $\lambda>0$ such that
\[
    g^2_{\Sc}=\lambda g^1_{\Sc}.
\]
On the quaternic multivector model, there is a natural choice of metric $g_{\Sc}$ which we briefly describe. On $\bC$, we let $g_{S_1}$ be the real part of the standard Hermitian metric $g_{\bC}(z,w)=\overline{z}w$ on $\bC$. On $\bH$, we take $g_{S_2}$ and $g_{S_3}$ to be the real part of the standard quaternic metric on $\bH$ $g_{\bH}\lp q,w\rp=\overline{q}w$. On $\wedge_{\bH}\bH$ we take $g_{S_4}$ to be the real part of the quaternic exterior power of the standard quaternic metric on $\bH$, $\wedge_{\bH}g_{\bH}\lp \lp q_1,q_2\rp,\lp w_1,w_2\rp\rp=\overline{q_1}w_1+\overline{q}_2w_2$. For a general $S_{8k+r}$, $0\leq r\leq 8$, we take $g_{S_r}$ to be tensor products of the corresponding metrics on $S_i$ defined above.\par
Fix an irreducible spinor module $S_n$ for $\bR^n$ along with a \emph{spin metric} $g_{S_n}$ satisfying (II) above. If we are given a module with such structure, denoted as
\[
   C\ell^{\bullet}(V,g)\circlearrowright(\Sc_V,g_{\Sc})\circlearrowleft\bK^{\bullet}_n, 
\]
then we define the \textbf{\emph{spin coordinate systems of $(V,g,\cal_{\bR},\Sc_V)$}} as the set of isometries $\Phi:(S_n,g_{S_n})\rightarrow (\Sc_V,g_{\Sc})$ which are right $\bK^0_n$-linear and such that there exists a \textbf{\emph{unique}} $\phi\in P_{SO}(V,g,\cal_{\bR})$ for which given any $v\in \bR^n$ and $s\in S_n$ one has
\[
    \Phi(c(v)\cdot s)=c(\phi(v))\cdot\Phi(s).
\]
That this set is non-empty follows from the classification of Clifford modules. We denote the set of spin coordinate systems of $(V,g,\cal_{\bR},\Sc_V)$ as $P_{Spin}(V,g,\cal_{\bR},\Sc_V)$. When no confusion can arise, we will omit the extra details from our notation of $P_{Spin}(V,g,\cal_{\bR},\Sc_V)$. Note that for any given $g\in Spin(n)$, one has that $\Phi\circ g\cdot\in P_{Spin}(V)$ and thus $P_{Spin}(V)$ admits a natural \emph{free right action} of $Spin(n)$
\begin{align*}
    P_{Spin}(V,g,\cal_{\bR},\Sc_V)\times Spin(n)&\longrightarrow P_{Spin}(V,g,\cal_{\bR},\Sc_V)\\
    (\Phi,g)&\longmapsto \Phi\circ g\cdot.
\end{align*}
Furthermore, note that $P_{Spin}(V,g)$ admits a natural map
\begin{align*}
    p:P_{Spin}(V,g,\cal_{\bR},\Sc_V)&\longrightarrow P_{SO}(V,g,\cal_{\bR})\\
    \Phi&\longmapsto \phi
\end{align*}
which is compatible with the double covering $Ad:Spin(n)\rightarrow SO(n)$ 
\[
    p(\Phi\circ g)=p(\Phi)\circ Ad(g),\quad \text{for all }\Phi\in P_{Spin}(V,g,\cal_{\bR},\Sc_V)\text{ and all }g\in Spin(n).
\]
To see that the map $p$ is $2$-to-$1$ suppose that $\Phi,\Psi\in P_{Spin}(V,g)$ are such that $p(\Phi)=p(\Psi)$. Then one has that 
\[
    \Psi^{-1}\Phi:(S_n,g_{S_n})\rightarrow (S_n,g_{S_n})
\]
such that for all $v\in C\ell_n$, $s\in S_n$, and $a\in \bK^0_n$, one has
\[
    \Psi^{-1}\Phi(c(v)\cdot s)=c(v)\cdot\Psi^{-1}\Phi(s),\quad \Psi^{-1}\Phi(s\cdot a)=\Psi^{-1}\Phi(s)\cdot a.
\]
The only such transformations of $(S_n,g_{S_n})$ which satisfy the above relations are $\pm I_{S_n}$, for $I_{S_n}$ the identity operator of $S_n$. Thus any given preimage $p^{-1}(\phi)$ is a $2$-point set. To see that $p$ is \emph{surjective} consider a fixed $\phi\in P_{SO}(V,g,\cal_{\bR})$ and a fixed $\Psi\in P_{Spin}(V,g,\cal_{\bR},\Sc_V)$.Then $p(\Psi)=\psi$ and there exists a unique $R\in SO(n)$ such that $\psi\circ R=\phi$. Then if $g\in Spin(n)$ is chosen so that $Ad(g)=R$, then $p(\Psi\circ g)=p(\Psi)\circ Ad(g)=\phi\circ R=\psi$. Thus the map $p$ is indeed $2$-to-$1$. \par 
What the above computation has shown is that if we were to consider the spin coordinate systems for $S_n$ then we have 
\[
    P_{Spin}(\bR^n,g_n,\cal_{\bR},S_n)=Spin(n),
\]
i.e. a linear isometry of $(S_n,g_{S_n})$ which commutes with the right action of $\bK^0_n$ and commutes with the left action of $C\ell_n$ \emph{up to an oriented isometry of $\bR^n$} must be a left action by an element $g\in Spin(n)$. From a certain perspective this is a geometric characterization of spin transformations in the sense that special orthogonal transformations are those which are linear isometries of $(\bR^n,g_n)$ which preserve the orientation.\par
We have yet to check that the action of $Spin(n)$ on $P_{Spin}(V,g,\cal_{\bR},\Sc_V)$ is \emph{transitive}. Moreover, we have also not considered a topology and smooth structure for $P_{Spin}(V,g,\cal_{\bR},\Sc_V)$ and shown that the action of $Spin(n)$ is \emph{smooth}. The choice of smooth structure is the natural choice: if we fix $\Phi\in P_{Spin}(V,g,\cal_{\bR},\Sc_V)$, then the map
\begin{align*}
    F_{\Phi}:P_{Spin}(V,g,\cal_{\bR},\Sc_V)&\rightarrow P_{Spin}(\bR^n,g_n,\cal_{\bR},S_n)\\
    \Psi&\mapsto \Psi^{-1}\Phi
\end{align*}
is a set bijection which is equivariant with respect to the right action of $Spin(n)$. Since $Spin(n)$ act transitively on $P_{Spin}(\bR^n,g_n,\cal_{\bR},S_n)$, we have that it also acts transitively on $P_{Spin}(V,g,\cal_{\bR},\Sc_V)$. We use the map $F_{\Phi}$ to equip $P_{Spin}(V,g,\cal_{\bR},\Sc_V)$ with the smooth structure of $P_{Spin}(\bR^n,g_n,\cal_{\bR},S_n)$. If we were to choose a different $\Psi\in P_{Spin}(V,g,\cal_{\bR},\Sc)$ then there exists a unique $g\in Spin(n)$ such that
\[
    F_{\Psi}\circ F_{\Phi}^{-1}(h)=g\cdot h \quad \text{for }h\in P_{Spin}(\bR^n,g_n,\cal_{\bR},S_n)
\]
and which implies smooth structure given by $F_{\Psi}$ and $F_{\Phi}$ agree. We see from the defining smooth structure that the map 
\[
    p:P_{Spin}(V,g,\cal_{\bR},\Sc_V)\rightarrow P_{SO}(V,g,\cal_{\bR})
\]
is smooth. Thus we have obtained our model of spin coordinate systems $P_{Spin}(V,g,\cal_{\bR},\Sc_V)$.
\subsection{Spin Structures on Smooth Vector Bundles}
We now wish to extend the notion of spin coordinate systems of $(V,g,\cal_{\bR},\Sc_V)$ to that of a real smooth vector bundle $E\rightarrow M$ over a smooth manifold. We will assume that $E\rightarrow M$ is equipped with a Euclidean fiber metric $g_E$ and has a global orientation $\cal_{\bR}$. Recall that $(E,g_E,\cal_{\bR})\rightarrow M$  \textbf{\emph{admits a spin structure}} if there exists a principal $Spin(n)$-bundle $P\circlearrowleft Spin(n)\rightarrow M$ and a fibered double covering
\[
    p:P\rightarrow P_{SO}(E,g_E,\cal_{\bR})
\]
such that for all $g\in Spin(n)$ one has $p(\Phi\cdot g)=p(\Phi)\cdot Ad(g)$. It is known that in the $Spin^c$ setting that a $Spin^c$-structure is equivalent to a bundle of complex Hermitian vector spaces $(\Sc,h)\rightarrow M$ which admits an irreducible left action by the bundle of complexified Clifford algebras $\bC\ell(E,g_E)\rightarrow M$ (Question 4.36 in \cite{Roe}). In this case, the analogue in the real setting is given below. We let $\underline{\bK_n}$ denote the trivial bundle of division algebras over $M$.
\begin{theorem}
    An oriented Euclidean vector bundle $(E,g_E,\cal_{\bR})\rightarrow M$ of rank $n$ admits a spin structure if and only if there is a real Euclidean vector bundle $(\Sc_E,g_{\Sc})\rightarrow M$ which is a bundle of $C\ell(E,g_E)-\underline{\bK_n}$ bimodules as well as a bundle of $C\ell^0(E,g_E)-\underline{\bK^0_n}$-bimodules for which the Clifford action is irreducible and the Clifford action by vectors in $E$ and right multiplication by $a\in\emph{Im}\underline{\bK_n}$ is skew-adjoint.
\end{theorem}
\begin{proof}
    Suppose that $(E,g_E,\cal_{\bR})\rightarrow M$ admits a spin lift $P\circlearrowleft Spin(n)\rightarrow P_{SO}(E,g_E,\cal_{\bR})\circlearrowleft SO(n)$. If $S_n$ denotes the real irreducible module constructed for $C\ell_n$ above, note that the spin representation $\rho :Spin(n)\rightarrow GL(S_n,\bR)$ is given by restriction of the action $C\ell_n\circlearrowright S_n$ to $Spin(n)$. In particular, the spin representation inherits all of the extra bells and whistles that the Clifford action has, i.e. commutativity with the right action by $\bK_n$ and the real metric $g_{S_n}$ for which the left action by vectors in $\bR^n$ and left action by elements in $\text{Im}\bK_n$ is skew-adjoint. Moreover, note that $Spin(n)\subseteq C\ell^0_n$, hence if $C\ell^0_n$ has a larger right intertwiner algebra $\bK^0_n$, then the action $Spin(n)\circlearrowright S_n$ will be compatible with the right action by $\bK^0_n$ as well. Since the right action $C\ell_n\circlearrowright S_n$ is compatible with the respective spin representations, this action descends to an action
    \[
        P\times_{Ad\rho} C\ell_n\circlearrowright (P\times_{\rho}S_n,g_{S_n})\circlearrowleft \underline{\bK_n},\quad P\times_{Ad\rho}C\ell^0_n\circlearrowright P\times_{\rho}S_n\circlearrowleft \underline{\bK^0_n}
    \]
    Because $C\ell(E,g_E)\simeq P\times_{Ad\rho}C\ell_n$ as a bundle of algebras, we have the asserted spinor module.\par
    Conversely, suppose we are given such a spinor module 
    \[
        C\ell(E,g_E)\circlearrowright (\Sc_E,g_{\Sc})\circlearrowleft\underline{\bK_n},\quad C\ell^0(E,g_E)\circlearrowright \Sc_E\circlearrowright \underline{\bK^0_n}
    \]
    We define the \textbf{\emph{principal bundle of spin coordinate systems for $(E,g_E,\cal_{\bR},\Sc_E)$}} to be the disjiont union
    \[
    P_{Spin}(E,g_E,\cal_{\bR},\Sc_E):=\bigsqcup_{m\in M}P_{Spin}(E_m,g_{E_m},\cal_{\bR},\Sc_{E_m}).
    \]
    The smooth structure is the natural choice and this fibration $P_{Spin}(E,g_E,\cal_{\bR},\Sc_E)\rightarrow M$ admits a right free and fibered transitive action of $Spin(n)$ such that the natural map
    \[
        p:P_{Spin}(E,g_E,\cal_{\bR},\Sc_E)\rightarrow P_{SO}(E,g_E,\cal_{\bR})
    \]
    is a smooth, fibered double covering which is compatible with the double covering $Ad:Spin(n)\rightarrow SO(n)$. Thus, we have obtained the desired spin structure $P_{Spin}(E,g_E,\cal_{\bR},\Sc_E)$.
\end{proof}
Given a spin lift $P\circlearrowleft Spin(n)\rightarrow P_{SO}\lp E,g_E,\cal_{\bR}\rp$, one can check that the map
\begin{align*}
    P&\longrightarrow P_{Spin}\lp E,g_E,\cal_{\bR},P\times_{\rho}S_n\rp\\
    p&\longmapsto \lp s\rightarrow [(p,s)]\rp
\end{align*}
where $[p,s]$ denotes the equivalence class in the quotient $P\times_{\rho}S_n$, gives an isomorphism of principal $Spin(n)$-bundles. Thus, what we have really shown is that \emph{every principal $Spin$ lift of $P_{SO}\lp E,g_E,\cal_{\bR}\rp$ is, up to isomorphism, of the type we have described. i.e. a bundle of spin coordinate systems for a Euclidean bundle of real spinors $\lp\Sc_E,g_{\Sc}\rp\rightarrow M$}.
\subsection{Constructing Spinor Bundles}
Suppose that $M=\cup_{\alpha\in A}U_{\alpha}$ where all finite intersections of the $U_{\alpha}$ are contractible, and we are also given explicit sections $e_{\alpha}\in \Gamma\lp P_{SO}(E,g_E,\cal_{\bR})\big|_{U_\alpha}\rp$. The $e_{\alpha}$ generate an oriented Euclidean bundle isomorphism
\[
    \lp E,g_E,\cal_{\bR}\rp\big|_{U_{\alpha}}\simeq_{e_{\alpha}} U_{\alpha}\times\lp\bR^r,g_r,\cal_{\bR}\rp
\]
which induce a spinor bundle for $C\ell\lp E,g_E\rp$ along $U_{\alpha}$
\[
    \lp C\ell^{\bullet}\lp E,g_E\rp\rp\big|_{U_{\alpha}}\circlearrowright_{e_{\alpha}}U_{\alpha}\times\lp S_r,g_{S_r}\rp\circlearrowleft\bK^{\bullet}_r.
\]
The transition functions $g_{\alpha\beta}:U_{\alpha}\cap U_{\beta}\rightarrow SO(r)$ have \emph{two possible lifts}
\[
    \pm \tilde{g}_{\alpha\beta}:U_{\alpha}\cap U_{\beta}\rightarrow Spin(r).
\]
Once a choice of lift has been made, one can use the $\tilde{g}_{\alpha\beta}$ as transition data for $U_{\alpha}\cap U_{\beta}\times \lp S_r,g_{S_r}\rp$. In general, what one will find is that the transition data \emph{almost} satisfy the cocylce condition
\[
    \tilde{g}_{\alpha\beta}\tilde{g}_{\beta\gamma}=\pm \tilde{g}_{\alpha\gamma}.
\]
What one then obtains is \emph{\v{C}ech $2$-cochain} with coefficients in $\{\pm 1\}$,\newline $w_2\lp E,\{ U_{\alpha}\}\rp\in C^2\lp \{U_{\alpha}\},\{\pm1\}\rp$ defined by
\[
    w_2\lp E,\{U_{\alpha}\}\rp\lp U_{\beta\gamma\delta}\rp:=\tilde{g}_{\beta\gamma}\tilde{g}_{\gamma\delta}\tilde{g}_{\beta\delta}^{-1}.
\]
One can check that this is in fact a $2$-cocycle. In general, the \v{C}ech cohomology class of this cocylce is the \emph{second Steifel-Whitney class} of the Euclidean vector bundle $(E,g_E)$ and is non-trivial in general. In the case when $w_2\lp E,\{ U_{\alpha}\}\rp$ is \emph{exact}, i.e. when $w_2\lp E,\{ U_{\alpha}\}\rp=\delta \tau$ for $\tau\in C^1\lp\{ U_{\alpha}\},\{\pm 1\}\rp$, then this allows us to \emph{adjust} our choice of patching data for $U_{\alpha}\times\lp S_r,g_{S_r}\rp$ along $U_{\alpha}\cap U_{\beta}$ by assigning
\[
    \tau\lp U_{\alpha\beta}\rp\tilde{g}_{\alpha\beta}:U_{\alpha}\cap U_{\beta}\rightarrow Spin(r)
\]
as the new patching data. One now checks that $w_2\lp E,\{U_{\alpha}\}\rp=\delta\tau$ ensures that the adjusted patching data satisfies the cocylce condition.\par
An alternative perspective to the gluing of local spinor modules can be achieved through the use of $\bR$-gerbes. As above, each local oriented Euclidean framing $e_{\alpha}$ of $P_{SO}\lp E,g_E,\cal_{\bR}\rp$ gives us an associated spinor module
\[
    C\ell^{\bullet}\lp E,g_E\rp\big|_{U_{\alpha}}\circlearrowright_{e_{\alpha}}\lp\Sc_{e_{\alpha}},g_{\Sc}\rp\circlearrowleft\bK^{\bullet}_r.
\]
On a given intersection $U_{\alpha\beta}$, we obtain a real line bundle
\[
    L_{\alpha\beta}:=\End_{C\ell,\bK^0_r}\lp\Sc_{e_{\alpha}},\Sc_{e_{\beta}}\rp
\]
and on a three fold intersection $U_{\alpha\beta\gamma}$, composition of linear maps gives an isomorphism 
\[
    c_{\alpha\beta\gamma}:L_{\alpha\beta}\big|_{U_{\alpha\beta\gamma}}\otimes_{\bR}L_{\alpha\beta}\big|_{U_{\alpha\beta\gamma}}\xlongrightarrow{\sim}L_{\alpha\gamma}\big|_{U_{\alpha\beta\gamma}}.
\]
We denote this family of $\bR$-line bundles by
\[
    \End_{C\ell,\bK_r^0}\lp\Sc_{\cdot},\Sc_{\cdot}\rp\rightarrow U_{\cdot,\cdot}
\]
and call this the \textbf{\emph{spin orientation $\bR$-gerbe associated to the trivializing cover $\{ U_{\alpha}\}$}}. An $\bR$-gerbe $L_{\cdot,\cdot}\rightarrow U_{\cdot,\cdot}$ admits a \emph{global non-vanishing section} if each $L_{\alpha\beta}\rightarrow U_{\alpha\beta}$ admits a non-vanishing section $s_{\alpha\beta}$ such that
\[
    c_{\alpha\beta\gamma}\lp s_{\alpha\beta}\otimes s_{\beta\gamma}\rp=s_{\beta\gamma}.
\]
A choice of non-vanishing section $s_{\alpha\beta}\in\End_{C\ell,\bK^0_r}\lp\Sc_{\alpha},\Sc_{\beta}\rp$ is equivalent, upon rescalling if necessary, to a choice of lifting $\tilde{g}_{\alpha\beta}:U_{\alpha\beta}\rightarrow Spin(r)$ of the transition date $g_{\alpha\beta}:U_{\alpha\beta}\rightarrow SO(r)$. Thus, we see that \emph{the oriented Euclidean bundle $(E,g_E,\cal_{\bR})$ admits a spin structure if and only if the spin orientation $\bR$-gerbe associated to the trivializing cover $\{ U_{\alpha}\}$ admits a global non-vanishing section}. One should compare this viewpoint on the existence of spin structures to the fact that a vector bundle $E\rightarrow M$ admits an orientation if and only if the $\bR$-orientation line bundle $\wedge^{top}_{\bR}E\rightarrow M$ admits a global section.\par
There are instances when one is working with a cover $\{U_{\alpha}\}$ which is trivializing but whose intersections may not all be contractible. We can still construct the $\bR$-gerbe $\End_{Cl,\bK^0_r}\lp\Sc_{\cdot},\Sc_{\cdot}\rp\rightarrow U_{\cdot,\cdot}$ however what can occur now that is that for a given $U_{\alpha,\beta}$, the line bundle $\End_{C\ell,\bK^0_r}\lp \Sc_{\alpha},\Sc_{\beta}\rp\rightarrow U_{\alpha\beta}$ may not admit a global section. Suppose however that we know in advance that our bundle $\lp E,g_E,\cal_{\bR}\rp\rightarrow M$ admits a global spinor bundle
\[
    C\ell^{\bullet}\lp E,g_E\rp\circlearrowright\lp\Sc_E,g_{\Sc}\rp\circlearrowleft\bK^{\bullet}_r
\]
and that each $U_{\alpha}$ is contractible. Then on each $U_{\alpha}$, the line bundle
\[
    \End_{C\ell,\bK^0_r}\lp \Sc_E\big|_{U_{\alpha}},\Sc_{\alpha}\rp\rightarrow U_{\alpha}
\]
is trivial. Moreover along a given intersection $U_{\alpha\beta}$ we have an isomorphism of line bundles
\begin{align*}
    \End_{C\ell,\bK^0_r}\lp\Sc_{\alpha},\Sc_E\big|_{U_{\alpha}}\rp\big|_{U_{\alpha\beta}}\otimes_{\bR}\End_{C\ell,\bK^0_r}\lp\Sc_E\big|_{U_{\beta}},\Sc_{\beta}\rp\big|_{U_{\alpha\beta}}\xrightarrow{\sim}\End_{C\ell,\bK^0_r}\lp\Sc_{\alpha},\Sc_{\beta}\rp.
\end{align*}
Thus, we have the following observation: \emph{Suppose $\lp E,g_E,\cal_{\bR}\rp\rightarrow M$ is a spin oriented Euclidean bundle and $\{U_{\alpha}\}$ is a trivializing open cover with each $U_{\alpha}$ contractible. Then for any choice of local spinor module $\Sc_{\alpha}\rightarrow U_{\alpha}$, the line bundles $\End_{C\ell,\bK^0_r}\lp\Sc_{\alpha},\Sc_{\beta}\rp\rightarrow U_{\alpha\beta}$ must be trivial along $U_{\alpha\beta}$ for all $\alpha,\beta$}. This produces a useful criteria for determining when an oriented Euclidean vector bundle \emph{does not} admit a global bundle of spinors.
\subsection{Spinors and Real Line Bundles}
Recall that for a given orientable Euclidean vector bundle $\lp E,g_E\rp\rightarrow M$, the set of orientations is acted on freely and transitively by the \emph{zeroth cohomology} of $M$ with $\bZ_2$-coefficients $H^0\lp M,\bZ_2\rp$. Similarly, there is typically not a \emph{unique} choice of spin structure for an oriented Euclidean bundle $\lp E,g_E,\cal_{\bR}\rp\rightarrow M$, instead the isomorphism classes of possible spin structures are acted on freely and transitively by the \emph{first cohomology} with $\bZ_2$-coefficients $H^1\lp M,\bZ_2\rp$ when $M$ is a finite CW complex. We give a brief treatment of this fact from the perspective of spinor bundles. Suppose that $\Sc_E\rightarrow M$ is a fixed real spinor bundle with all of the bells and whistles. Suppose that $\lp L,g_L\rp\rightarrow M$ is a Euclidean line bundle over $M$. Then we can \emph{twist $\Sc_E$ by $L$} 
\[
    \lp L,g_L\rp\otimes_{\bR}\lp \Sc_E,g_{\Sc}\rp\rightarrow M.
\]
While there are instances when this bundle is isomorphic to the original $\Sc_E$ at the level of $\bR$-vector bundles, if we consider the action of $C\ell(E,g_E)$ on $L\otimes_{\bR}\Sc_E$ defined by
\[
    c(v)\cdot a\otimes s:=a\otimes c(v)\cdot s
\]
then this gives us a new real spinor bundle. Moreover, one has 
\[
    L\simeq \End_{C\ell,\bK_n^0}\lp L\otimes_{\bR} \Sc_E,\Sc_E\rp.
\]
Thus there will be a global identification of $L\otimes_{\bR}\Sc_E$ with $\Sc_E$ as bundles of $C\ell^{\bullet}(E,g_E)-\underline{\bK_n^{\bullet}}$ bimodules if and only if $L$ is trivial. Under the natural identification of the \emph{real Picard group} with $H^1\lp M,\bZ_2\rp$ via the first Steifel-Whitney class
\[
    w_1:\lp \text{Vect}^1_{\bR}\lp M\rp,\otimes_{\bR}\rp\xrightarrow{\sim}H^1\lp M,\bZ_2\rp
\]
we have recovered the action of the real Picard group on the set of spin structures of $\lp E,g_E,\cal_{\bR}\rp$. We leave it to the reader to check that this action is free and transitive. 
\section{Novelties}
We now collect a few novel low dimensional constructions from our angle and include a discussion on the tensor product of spinor fields.
\subsection{The Spin Geometry of a Surface}
In the same way that much of the intuition for the \emph{Levi-Civita connection} is derived from studying surfaces $\Sigma$ embedded in $\bR^3$, we will use our explicit spinor representations to unravel the spin geometry of an oriented $\lp\Sigma,\cal\rp\hookrightarrow \bR^3$. Thus the vector bundle in consideration will be $T\Sigma\rightarrow \Sigma$ and the metric will be the \emph{first fundamental form} of $\Sigma$, denoted $g_{\Sigma}$. Firstly, note that the orientation of $\Sigma$ gives a global normal vector field $\nu:\Sigma\rightarrow\bR^3$ which gives a global trivialization of
\begin{align*}
    \wedge^1T\Sigma\oplus\wedge^2 T\Sigma&\xlongrightarrow{\sim}\underline{\bR^3}\\
    \lp \vec{v},a\boldsymbol{\nu}\rp&\longmapsto \vec{v}+a\nu.
\end{align*}
Note that we may treat $\bR^3\simeq \text{Im}\bH$. Given that the metric $g_{\Sigma}$ is simply the restriction of $g_{\bR^3}$ to $T\Sigma$, the above map gives an identification of bundles of algebras
\[
    C\ell\lp T\Sigma,g_{\Sigma}\rp\simeq \underline{\bH}.
\]
Thus if we take $\Sc_{\Sigma}:=\underline{\bH}$ with the standard quaternic metric 
\[
    g_{\bH}\lp q,w\rp:=\overline{q}\cdot w,
\]
then we have almost obtained a real spinor module
\[
    C\ell\lp T\Sigma,g_{\Sigma}\rp\circlearrowright \lp\underline{\bH},\text{Re}g_{\bH}\rp\circlearrowleft\bH.
\]
In dimension $n=2$, the intertwiner algebra $\bK^0_2\simeq M_2\lp\bC\rp$, thus there should be a natural action of $M_2\lp\bC\rp$ on $\underline{\bH}$ which commutes with the right action of $C\ell^0(T\Sigma,g_{\Sigma})$. Given a point $p\in\Sigma$, note that
\[
    C\ell^0(T\Sigma,g_{\Sigma})_p\simeq \text{Span}\{1,\nu_p\}\simeq\bC.
\]
Because the action of $A\in M_2\lp\bC\rp$ on $\bC^2\simeq\bH$ is of the form
\[
    A\cdot q=k\cdot q\cdot w+q\cdot w',\quad w,w'\in\bH,
\]
we see that the action of $A\in M_2\lp\bC\rp$ on $\underline{\bH}_p$ is
\[
    A\cdot_p q:=\nu_p\cdot q\cdot w+q\cdot w'.
\]
This gives us a spinor bundle with the desired properties
\[
    C\ell^{\bullet}\lp T\Sigma,g_{\Sigma}\rp\circlearrowright\lp\underline{\bH},\text{Re}g_{\bH}\rp\circlearrowleft\bK^{\bullet}_2.
\]
We obtain an induced principal $Spin(2)$ lift of the special orthogonal frame bundle
\[
    P_{Spin}\lp\underline{\bH},\text{Re}g_{\bH}\rp\circlearrowleft Spin(2)\rightarrow P_{SO}\lp T\Sigma,g_{\Sigma},\cal_{\bR}\rp\circlearrowleft SO(2).
\]
Suppose that $x\in\Sigma$ is fixed and we are given an oriented orthonormal frame $e_1,e_2\in T_x\Sigma$. Then there is a unique $R\in SO(3)$ such that 
\[
    R\cdot[i\hspace{.1cm} j\hspace{.1cm}k]=[e_1\hspace{.1cm}e_2\hspace{.1cm}\nu_x],
\]
where $i,j,k$ are the standard quaternic basis for $\bR^3=\text{Im}\bH$. Let $g\in \mathbb{S}^3$ be viewed as a unit quaterion such that for every $q\in \text{Im}\bH$ one has
\[
    g\cdot q\cdot g^{-1}=Rq.
\]
Then the \emph{spin coordinate system corresponding to $g$} is the isometry
\begin{align*}
    \Phi_g:\lp\bH,\text{Re}g_{\bH}\rp&\xrightarrow{\sim}\lp\bH,\text{Re}g_{\bH}\rp\\
    q&\longmapsto g\cdot q.
\end{align*}
Given $v=ai+bj$, we have
\[
    \Phi_g\lp v\cdot q\rp=g\cdot v\cdot q=g\cdot v\cdot g^{-1}\cdot g\cdot q=\lp ae_1+be_2\rp\Phi_g(v)
\]
and for $w\in\bH$ and $A\in M_2\lp\bC\rp$ expressed as $A=k_La_R+b_R$ for $a,b\in\bH$, we have
\[
    \Phi_g(q\cdot w)=\Phi_g(q)w,\quad \Phi_g\lp\lp k_La_R+b_R\rp w \rp=\lp\nu_La_R+b_R\rp\Phi_g\lp w\rp
\]
which shows that $\Phi_g$ is indeed a spin coordinate system. Note that the other choice of spin coordinate system covering the frame $\{ e_1,e_2\}$ would be $\Phi_{-g}$.
\subsection{Parallel Translation of Spinors}
We briefly recall E. Cartan's \emph{moving frame} approach to the Levi-Civita connection. Suppose $\gamma:[a,b]\rightarrow \Sigma$ is a smooth curve and $t_0\in(a,b)$. Then if $V_t\in\Gamma\lp \gamma^*T\Sigma\rp$, the \emph{covariant derivative of $V_t$ along $\gamma$} is defined as
\[  
    D_tV_t:=P^{T\Sigma}\frac{d}{dt}V_t.
\]
In the above, we are treating $V_t:[a,b]\rightarrow \bR^3$ to take $\frac{d}{dt}$ in the ordinary sense, and $P^{T\Sigma}$ is the projection of $\bR^3$ to $T\Sigma$. We say that $V_t$ is \emph{parallel along $\gamma$} if $D_tV_t\equiv 0$. Given $t_0\in(a,b)$ and $v\in T_{\gamma(t_0)}\Sigma$, ODE theory grants the existence and uniqueness of a \emph{parallel transport} of $v$ along $\gamma$, i.e. $V_t\in\Gamma\lp\gamma^*T\Sigma\rp$ with
\[
    D_tV_t\equiv 0,\quad V_{t_0}=v.
\]
For a given oriented orthonormal frame $\{ e_1,e_2\}$ in $T_{\gamma(t_0)}\Sigma$, the Levi-Civita connection of $\lp \Sigma,g_{\Sigma}\rp$ instructs us on how to parallel translate the initial frame along $\gamma$ giving us a parallel section $\{ e_1(t),e_2(t)\}$ of $P_{SO}\lp T\Sigma,g_{\Sigma},\circlearrowright_{\bR}\rp$. If we view $\{ e_1(t),e_2(t)\}$ as a curve in $P_{SO}\lp T\Sigma,g_{\Sigma},\circlearrowleft\rp$, then this curve defines the \emph{horizontal lift of $\gamma$ through the frame $\{ e_1,e_2\}$} and we will denote this as $\gamma^{H}_{\{e_1,e_2\}}:[a,b]\rightarrow P_{SO}\lp T\Sigma,g_{\Sigma},\circlearrowleft\rp$.\par
As above, let $R_{\{e_1,e_2\}}:[a,b]\rightarrow SO(3)$ be be the unique smooth map which satisfies
\[
    R_{\{e_1,e_2\}}(t)[i\hspace{.1cm}j\hspace{.1cm}k]=[e_1(t)\hspace{.1cm}e_2(t)\hspace{.1cm}\nu_{\gamma(t)}]
\]
Then by the \emph{path lifting property} of the double covering $\mathbb{S}^3\rightarrow SO(3)$, we have two possible smooth choices of path lifting $\pm g_t:[a,b]\rightarrow\mathbb{S}^3$ such that for $v\in\bR^3$
\[
    R_{\{ e_1,e_2\}}(t)(v)=\pm g_t\cdot v\cdot \lp\pm g_t\rp^{-1}.
\]
Making a choice of $\pm$, gives a \emph{parallel spin coordinate system along $\gamma$}, $\Phi_{g(t)}:[a,b]\rightarrow P_{Spin}\lp\underline{\bH},\text{Re}g_{\bH}\rp$. If we are given a fixed spinor $q\in\underline{\bH}_{\gamma(t_0)}$, then there is a unique $q_0\in\bH$ such that
\[
    \Phi_{g(t_0)}\lp q_0\rp=q.
\]
The \emph{parallel transport of the spinor $v$ along $\gamma$} is then defined to be
\[
    q(t):=\Phi_{g(t)}\lp q_0\rp.
\]
If we view the parallel propagation of spin coordinate systems as a method for lifting paths from $\Sigma$ to $P_{Spin}\lp \underline{\bH},\text{Re}g_{\bH}\rp$, then this defines the \textbf{\emph{spin Levi-Civita connection}} on spinor fields.
\begin{example*}
Consider the unit sphere $\bS^2=\bigl\{(x,y,z)\in\bR^3\big|\quad x^2+y^2+z^2=1\bigr\}$. We consider the geodesic $\gamma(t)=\lp \sin(2\pi t),0,\cos(2\pi t)\rp$ with initial framing $\{ e_1(0),e_2(0)\}=\{ i,j\}$. Then for $\gamma^H_{\{e_1,e_2\}}(t)=\{ e_1(t),e_2(t)\}$ we have
\[
    R_{\{ e_1,e_2\}}(t)=\begin{bmatrix}
        \cos(2\pi t) & 0 & \sin(2\pi t)\\
        0 & 1 & 0\\
        -\sin(2\pi t) & 0 & \cos(2 \pi t)
    \end{bmatrix}
\]
and we can choose as our lift $g:[0,1]\rightarrow\bS^3$ as
\[
    g(t)=\exp\lp \pi jt\rp.
\]
In this case, if we are given an initial spinor $q\in\underline{\bH}_{\gamma(0)}$ and $q=g_0\cdot q_0$, then the parallel propagation of $q$ along $\gamma$ will be
\[
    q(t)=\exp(\pi jt)\cdot q_0.
\] One rather interesting consequence of this, is that if we treat a tangent vector, such as $i$ at $\gamma(0)$ as a spinor, then the spin parallel transport will be
\[
    i(t)=\exp(\pi jt)\cdot i=\cos(\pi t)i-\sin(\pi t)k
\]
which does \emph{not} remain tangent to the surface.\qed\par
\end{example*}
The above calculation demonstrates that on a general Riemannian $2$-fold $\lp \Sigma,g_{\Sigma}\rp$, if we take the spinor bundle to be $\Sc_{\Sigma}=\wedge T\Sigma\simeq \underline{\bH}$, then a \emph{constant spinor} $q:\Sigma\rightarrow \bH$, $q(x)=q$ for $x\in\Sigma$, is not typically going to be \emph{parallel with respect to the spin Levi-Civita connection}. Even more drastic, the volume form $\nu$, although parallel with respect to the Levi-Civita connection, will in general not be parallel with respect to the spin Levi-Civita connection.\par
If the surface $\Sigma$ has non-trivial $H^1\lp\Sigma,\bZ_2\rp$, then the above spinor module will not be the only possible choice. Although $\Sigma$ may not be compact, we may still use 
\[
    H^1\lp\Sigma,\bZ_2\rp\simeq \lp \text{Vect}^1_{\bR}\lp\Sigma\rp,\otimes_{\bR}\rp.
\]
The latter is parameterized by the set of homotopy classes of maps $[\Sigma,\bR P^{\infty}]$, which by Sard's theorem may be reduced to the set of homotopy classes of maps $[\Sigma,\bR P^3]\simeq[\Sigma,SO(3)]$. Given a map $R:\Sigma\rightarrow SO(3)$, we may define a new action of $C\ell(\Sigma,g_{\Sigma})$ on $\underline{\bH}$ by
\begin{align*}
    C\ell\lp\Sigma,g_{\Sigma}\rp\times\underline{\bH}&\longrightarrow\underline{\bH}\\
    (v,q)&\longmapsto \lp Rv\rp\cdot q.
\end{align*}
\subsection{Products of Spinors}
It is well known that the existence of non-trivial parallel spinors give rise to interesting geometric structures such as the existence of Ricci-flat Kähler manifolds, and holnomoy reductions of the tangent bundle to $G_2$ and $Spin(7)$ (see Chapter $4$ of \cite{LM} for survey of results in this area). One of the main techniques in this direction is the ``squaring" operation on spinors. We discuss this operation now.\par
For a general rank $r$ oriented Euclidean vector bundle $\lp E,g_E,\circlearrowleft\rp\rightarrow M$ equipped with a spinor bundle
\[
    C\ell^{\bullet}\lp E,g_E\rp\circlearrowright\lp\Sc_E,g_{\Sc}\rp\circlearrowleft\bK^{\bullet}_r
\]
we have that
\[
    C\ell\lp E,g_E\rp\simeq \End^R_{\bK_r}\lp \Sc_E\rp, (\text{ $r\neq3$ mod 4}),\quad C\ell^0\lp E,g_E\rp\simeq \End^R_{\bK_r}\lp\Sc_E\rp, (\text{ $r=3$ mod 4} ).
\]
Note furthermore that in each various dimension, the spinor metric $g_{\Sc}$ can be extended to a fiber Hermitian, or fiber hyper-Hermitian metric when $\bK_r\simeq\bC$ or $\bK_r\simeq\bH$, denoted $g^{\bK_r}_{\Sc}$. In this case, if $s\in\Gamma\lp\Sc_E\rp$ then we let
\[
    \overline{s}:=g^{\bK_r}_{\Sc}\lp s,\cdot\rp\in\Gamma\lp\Hom^R_{\bK_r}\lp\Sc_E,\bK_r\rp\rp
\]
denote the \textbf{\emph{$\bK_r$-dual of $s$}}. Observe that $\Hom^R_{\bK_r}\lp\Sc_E,\bK_r\rp$ is naturally a left $\bK_r$-module. In this case, if we are given two spinor fields $s_1,s_2$, we may take their \emph{tensor product} and obtain 
\[
    s_1\otimes_{\bK_r}\overline{s}_2\in\End^R_{\bK_r}\lp \Sc_E\rp
\]
and thus there exists a unique mixed degree differential form $\omega\in\wedge_{\bR}E$ such that
\[
    c(\omega)=s_1\otimes_{\bK_r}\overline{s}_2.
\]
If $\lp E,g_E,\cal_{\bR}\rp\rightarrow M$ comes equipped with a metric compatible connection $\nabla^E$, then we let $\nabla^{\Sc_E}$ denote the associated spin connection on $\Sc_E$. The bundle $\Hom^R_{\bK_r}\lp \Sc_E,\bK^r\rp$ has a canonical connection induced by $\nabla^{\Sc_E}$, which we denote by $\nabla^{\Hom^R_{\bK_r}\lp\Sc_E,\bK_r\rp}$. Then we have
\[
    \nabla^{\wedge E}\lp s_1\otimes_{\bK_r}\overline{s}_2\rp=\nabla^{\Sc_E}s_1\otimes_{\bK_r}\overline{s}_2+s_1\otimes_{\bK^r}\nabla^{\Hom^R_{\bK_r}\lp\Sc_E,\bK_r\rp}\overline{s}_2.
\]
Suppose now that $(M,g)$ is an oriented Riemannian manifold which admits a global family of spinors 
\[
    C\ell^{\bullet}\lp TM,g\rp\circlearrowright\lp \Sc_{TM},g_{\Sc}\rp\circlearrowleft\bK^{\bullet}_r.
\]
In this case, if $\slashed{D}$ denotes the \textbf{\emph{spinor Dirac operator associated to the spin Levi-Civita connection}} given by
\[
    \slashed{D}=c\circ\nabla^{\Sc}
\]
then for spinor fields $s_1$ and $s_2$ we have
\[
    \lp\slashed{D}s_1\rp\otimes_{\bK_r}\overline{s}_2=\lp d+d^*\rp\lp s_1\otimes_{\bK_r}\overline{s}_2\rp-c\lp s_1\otimes_{\bK_r}\nabla^{\Hom^R_{\bK_r}\lp\Sc_M,\bK_r\rp}\overline{s}_2\rp.
\]
The above relationship between $\slashed{D}$ and $d+d^*$ given above, along with the fact that the operators $d$ and $d^*$ satisfy their own local fundamental theorem of calculus, gives us a macroscopic method for computing the Dirac operator. Recall that the \emph{flux of $\omega\in\Omega^kM$ along an oriented submanifold $\lp S,\cal\rp$ of dimension $n-k$} is defined as
\[
    \int^{\F}_{\lp S,\cal\rp}\omega:=\lp-1\rp^{k(n-k)}\int_{\lp S,\cal\rp}\star\omega
\]
and $d^*$ satisfies \emph{Gauss's theorem}
\[
    \int^{\F}_{\lp S,\cal\rp}d^*\omega=(-1)^k\int^{\F}_{\lp \partial S,\partial\cal\rp}\omega,\quad\omega\in\Omega^{n-k+1}M.
\]
\begin{proposition}
    Let $(M,g,\cal)$ be spin oriented with spinor bundle $\lp\Sc_M,g_{\Sc}\rp\rightarrow M$ and let $s\in\Gamma\lp\Sc_M\rp$ be parallel at the point $p\in M$. Suppose $(S,\cal_{\bR})\leqslant T_pM$ is a $k$-dimensional subspace of $T_pM$ and let $(S_r,\cal):=\exp_p\lp\{v\in S\big|\quad |v|<r\},\cal\rp$. If $dV_S\in\wedge T_pM$ denotes the volume multi-vector of $\lp S,\cal_{\bR}\rp$, then for a given spinor field $t$ we have
    \[
        \la\slashed{D}t\otimes_{\bK_r}\overline{s},dV_S\ra=\lim_{r\rightarrow 0^+}\frac{1}{\textup{Vol}(S_r)}\int_{\lp\partial S_r,\partial\cal\rp}t\otimes_{\bK_r}\overline{s}+\frac{(-1)^{n(k+1)}}{\textup{Vol}(S^{\perp}_r)}\int^{\F}_{\lp\partial S^{\perp}_r,\partial \cal^{\perp}\rp} t\otimes_{\bK_r}\overline{s}.
    \]
\end{proposition}
If we are given a locally defined $\nabla^{\Sc}$-parallel spinor field, we have
\[
    \lp\slashed{D}s_1\rp\otimes_{\bK_r}\overline{s}_2=\lp d+d^*\rp\lp s_1\otimes_{\bK_r}\overline{s}_2\rp.
\]
One should note that the existence of a parallel spinor field places \emph{strong conditions} on the metric $g$ as can be seen in \cite{Wan}. The most obvious obstruction is the scalar curvature via
the Lichnerowicz formula
\[
    \slashed{D}^2=\nabla^*\nabla+\frac{k}{4}.
\]
If there is a parallel spinor field along a neighborhood $U$ of $(M,g,\cal)$, then the scalar curvature $k$ must vanish along $U$.\par
The existence of a global parallel spinor field also gives a direct relationship between the spectral properties of $\slashed{D}$ and $d+d^*$.
\begin{proposition}
    Let $(M,g,\cal)$ be spin oriented with spinor bundle $\lp\Sc_M,g_{\Sc}\rp\rightarrow M$. If $\Sc_M$ admits a non-trivial parallel spinor $s$, then if $\lambda\in\bR$ is an eigenvalue of $\slashed{D}$, $\lambda$ is also an eigenvalue of $d+d^*$ and we have the inequality
    \[
        \textup{dim}_{\bR}\bK_r\cdot\textup{dim}_{\bK_r}E^{\slashed{D}}_{\lambda}\cdot\textup{dim}_{\bK_r}\H^{\slashed{D}}\leq \textup{dim}_{\bR}E^{d+d^*}_{\lambda}.
    \]
\end{proposition}
In dimensions $4k$, the spinor bundle is $\bH-\bZ_2$-graded or $\bR-\bZ_2$-graded and induces a grading on $\Hom^R_{\bK_r}\lp\Sc,\bK_r\rp$ which satisfies
\[
    \Sc_E^{(-1)^i}\otimes_{\bK_{4k}}\Hom^R_{\bK_{4k}}\lp\Sc_E,\bK_{4k}\rp^{(-1)^j}\rightarrow \bigoplus\wedge^{2l+(-1)^{i+j}}_{\bR}E.
\]
When $E=TM$ admits a non-trivial parallel spinor, we see that either $\text{dim}_{\bK_{4k}}\H^{\slashed{D},+}>0$ or $\text{dim}_{\bK_{4k}}\H^{\slashed{D},-}>0$. The above proposition then gives
\begin{align*}
    \text{dim}_{\bR}\bK_{4k}\cdot\text{dim}_{\bK_{4k}}\H^{\slashed{D},(-1)^i}\cdot\text{dim}_{\bK_r}\H^{\slashed{D},+}&\leq \dim_{\bR}\bigoplus H^{2i+(-1)^i}_{dR}(M,\bR),\quad \text{if }\text{dim}_{\bK_r}\H^{\slashed{D},+}>0\\
    \text{dim}_{\bR}\bK_{4k}\cdot\text{dim}_{\bK_{4k}}\H^{\slashed{D},(-1)^i}\cdot\text{dim}_{\bK_r}\H^{\slashed{D},-}&\leq \dim_{\bR}\bigoplus H^{2i+(-1)^i+1}_{dR}(M,\bR),\quad \text{if }\text{dim}_{\bK_r}\H^{\slashed{D},-}>0
\end{align*}
It is not clear if the above proposition is optimal. 
\subsection*{Spinors Along Oriented Hypersurfaces in $\bR^4$}
Suppose $\lp\Sigma,\cal\rp\hookrightarrow\bR^4$ is an oriented, regular submanifold of $\bR^4$, treat $\bR^4\simeq\bH$, and give $\Sigma$ the subspace metric $g_{\Sigma}$. If we let $\nu:\Sigma\rightarrow\bS^3$ denote the unit normal of $\lp\Sigma,\cal\rp$, then for a given $p\in\Sigma$, note that left quaternic multiplication by $\nu_p$ gives an isometric isomorphism
\[
    \nu_p\cdot_{\bH}:T_p\Sigma\xlongrightarrow{\sim}\text{Im}\bH.
\]
This gives us an action of $C\ell\lp T_p\Sigma,g_{\Sigma,p}\rp$ on the trivial quaternic bundle $\underline{\bH}$ via
\begin{align*}
    T_p\Sigma&\longrightarrow \End^R_{\bH}\lp\bH\rp\\
    v&\longmapsto\lp q\mapsto \nu_p\cdot v\cdot q\rp.
\end{align*}
If we use the standard quaternic metric on $\underline{\bH}$
\[
    g_{\bH}\lp q,w\rp:=\overline{q}\cdot w
\]
then we obtain a spinor module with all the bells and whistles
\[
    C\ell\lp T\Sigma,g_{\sigma}\rp\circlearrowright\lp\underline{\bH},\text{Re}g_{\bH}\rp\circlearrowleft\bH
\]
thus showing that every oriented hypersurface of $\bR^4$ admits a spin structure. In a similar vein as for oriented hypersurfaces in $\bR^3$, one can explicitly compute the Levi-Civita connection. A nice spin-off to this that left quaternic multiplication by $\nu$ gives a global trivialization of the tangent bundle of $\Sigma$.
\section{An Alternative Set of Representations}
At this point, one may inquire on whether or not one should consider thinking of spinors geometrically as quaternic multivectors in dimensions $=0$ (mod $4$), and some geometric cousin of these objects in other dimensions. The generality of our construction of spinor modules provides a method for probing this question; in particular we can ask if there is a distinct set of spinor modules which has a different \emph{geometric flavor} than the one presented with quaternions we have given above. We give an alternative construction of spinor modules in dimensions $1,2,3,\text{ and } 4$ which is based on decomposing the canonical representation of the Clifford algebra $C\ell_n$ on the exterior algebra $\wedge_{\bR}\bR^n$.
\subsection{An Alternative Set of Spinor Modules}
For $n=1$, we take as our irreducible module the Clifford algebra acting on the exterior algebra 
\[
    C\ell_1\circlearrowright\wedge_{\bR}\bR.
\]
In dimension $2$, we also take as our irreducible module the exterior algebra
\[
    C\ell_2\circlearrowright\wedge_{\bR}\bR^2.
\]\par
In dimension $3$, the action $C\ell_3\circlearrowright\wedge_{\bR}\bR^3$ is too large. We can slice this action in half by using the \emph{Hodge star operator $\star$} to identify $\wedge^0\bR^3\simeq_{\star}\wedge^3\bR^3$ and $\wedge^1\bR^3\simeq_{\star}\wedge^2\bR^2$. In which case this gives us the action
\[
    c_3(v)\lp\lambda,w\rp:=\lp -\la v,w\ra,\lambda v+\star v\wedge w\rp
\]
where $v,w\in\wedge^1\bR^3$ and $\lambda\in\wedge^0\bR^3$.\par
In dimension $4$, we can still identify $\wedge^0\bR^4\simeq_{\star}\wedge^4\bR^4$ and $\wedge^1\bR^4\simeq_{\star}\wedge^3\bR^4$. Since the splitting $\wedge^2\bR^4=\wedge^+\bR^4\oplus\wedge^-\bR^4$ is invariant under $\star$, we have to make a choice of summand in order to slice $\wedge\bR^4$ in half. Either choice will do, but we will choose the anti-self dual component $\wedge^-\bR^4$ and let $P_{\wedge^-}$ denote the projection onto the anti-self dual component. In this case
\[
    c_4(v)\lp\lambda,w,\tau\rp:=\lp-\la v,w\ra,\lambda v+\star v\wedge\tau-\iota_v\tau,P_{\wedge^-}v\wedge w\rp
\]
where $\lambda\in\wedge^0\bR^4$, $v,w\in\wedge^1\bR^4$, and $\tau\in\wedge^-\bR^4$.\par
If we are given a general $k$-multivector $\omega\in\wedge^k\bR^n$, the Hodge star satisfies
\[
    \omega\wedge_{\bR}\star\omega=|\omega|^2\boldsymbol{\nu}
\]
where $|\omega|^2$ is the $k$-volume of $\omega$ squared. From this perspective it seems that identifying a form with it's Hodge star is equivalent to taking a square root of the volume form. For this reason, we will call these representations the \textbf{\emph{Square Roots of Space}} model for spinors.\par
This method of constructing the irreducible representations of $C\ell_n$ by slicing $\wedge_{\bR}\bR^n$ in half using the hodge star operator clearly fails beyond dimension $4$, since 
\[
    \lp\text{dim}_{\bR}\wedge_{\bR}\bR^n\rp/2>\text{dim}_{\bR} S_n
\]
for $n>4$. However, our general recipe tells us that the better way to approach this method of constructing the square roots of space irreducible modules for $C\ell_n$ should be to first choose an oriented factoring 
\[
    \lp\bR^n,\cal_{\bR}\rp=\lp V_1,\cal_{\bR}\rp\oplus\cdots\lp V_k,\cal_{\bR}\rp\oplus\lp W,\cal_{\bR}\rp
\]
where $n=4k+r$, $\text{dim}_{\bR} V_j=4$ and $\text{dim}_{\bR} W=r$. From here one slices the forms within $\wedge_{\bR} V_j$ and $\wedge_{\bR}W$ using $\star$, and then takes appropriate tensor products of these square root representations. We will denote these representations by $\sqrt{\wedge_{\bR}\bR^n}$, and note that the choice of square root depends on the choice of square root necessarily depends on the choice of oriented splitting of $\lp\bR^n,\cal\rp$ given above. \par
%We further remark that for these representations one typically needs to \emph{choose} a choice of algebra isomorphism $\End_{C\ell_n}\sqrt{\wedge_{\bR}\bR^n}\simeq \bK_n$ and that this isomorphism is \emph{not} typically canonical. Once again, the situation for general $n$ boils down to understanding the case of $1\leq n\leq 4$. When $n=1$ or $n=2$, we may use the fact that $\wedge_{\bR}\bR^n$ is a \emph{bi-module} with respect to the Clifford action
%\[
%    C\ell_n\circlearrowright\wedge_{\bR}\bR^n\circlearrowleft C\ell_n
%\]
%which gives $\End_{C\ell_n}\wedge_{\bR}\bR^n\simeq C\ell_n$ for $n=1$ and $n=2$. For general $n$, we need to be Observe that when $n=3$, there is equally well a slicing of the right action $\wedge_{\bR}\bR^3\circlearrowleft C\ell_3$ along the Hodge star $\star$
%\[
%    \sqrt{\wedge_{\bR}\bR^3}\circlearrowleft C\ell_3
%\]
%which commutes with the left action $C\ell_3\circlearrowright\wedge_{\bR}\bR^3$. If we restrict the right action to $C\ell^0_3$, we obtain 
%\[
%    C\ell_3\circlearrowright\sqrt{\wedge_{\bR}\bR^3}\circlearrowleft C\ell^0_3
%\]
%which upon an identification $C\ell^0_3\simeq \bH$, gives the right quaternic action on $\sqrt{\wedge_{\bR}\bR^3}$.
%\par
These representations can be used as follows: Suppose that $\lp E,g_E,\circlearrowleft\rp\rightarrow M$ has rank $r\leq 4$. Then because in the above construction we only used the Hodge star operator $\star$ on multivectors, we can form the vector bundle $\sqrt{\wedge_{\bR}E}\rightarrow M$ which comes with an irreducible action of $C\ell(E,g_E)$
\[
    C\ell(E,g_E)\circlearrowright\sqrt{\wedge_{\bR}E}.
\]
However, this \emph{does not imply} that the oriented Euclidean bundle $(E,g_E,\circlearrowleft)\rightarrow M$ is spin. The problem in this case can be traced back to \emph{apriori} absence of the bells and whistles that a true bundle of spinors must have. To give an example, let $L\rightarrow\bR P^2$ denote the universal $\bR$-line bundle over $\bR P^2$. Then if we take $E:=L\oplus L$ and equip $E$ with an arbitrary metric $g_E$, then $E$ is indeed orientable since
\[
    \wedge^2_{\bR}E\simeq L\otimes_{\bR}L\simeq \underline{\bR}.
\]
However, the bundle of intertwiners of $\sqrt{\wedge_{\bR}E}$ is given by 
\[
    \End_{Cl}\lp\sqrt{\wedge_{\bR}E}\rp\simeq C\ell\lp E,g_E\rp\simeq \wedge_{\bR}E\simeq \underline{\bR^2}\oplus E
\]
which cannot be globally trivialized to $\underline{\bH}$ since this would imply $E$ is stably trivial, which it isn't.
\subsection{Yet Another Set of Representations}
For completeness, we include another set of spinor representations modeled on the Octonions. One can find these in \cite{LM}, \cite{E}, and in \cite{RB}. The first three representations we constructed 
\[
    C\ell_1\circlearrowright\bC,\quad C\ell_2\circlearrowright\bH,\quad C\ell_3\circlearrowright\bH,
\]
remain unchanged. If one loosens the definition of $\bR$-algebra and allow for non-associative product structures, then the next division algebra after the quaternions is the \textbf{\emph{Octonions}}, which are defined as
\[
    \bO:=\bH\oplus\bH
\]
with product and conjugation structure defined as
\[
    \lp a,b\rp\cdot\lp c,d\rp:=\lp ac-\overline{d}b,da+b\overline{c}\rp,\quad \overline{\lp a,b\rp}:=\lp \overline{a},-b\rp.
\]
The \textbf{\emph{real axis}} and \textbf{\emph{imaginary axis}} of $\bO$ are the subspaces
\[
    \text{Re}\bO:=\bigl\{\lp a,b\rp\in\bO\big|\text{ $a\in\bR$ and $b=0$}\bigr\},\quad \text{Im}\bO:=\bigl\{\lp a,b\rp\in\bO\big|\text{ $a\in\text{Im}\bH$ and $b\in\bH$}\bigr\}.
\]
Note that we can form the embedding
\begin{align*}
    c:\bH&\rightarrow \text{Im}\bO\\
    q&\mapsto\lp 0,q\rp
\end{align*}
which induces a map into operators on $\bO$
\begin{align*}
    R_c:\bH&\rightarrow\End_{\bR}\bO\\
    q&\mapsto\lp \lp a,b\rp\mapsto \lp a,b\rp\cdot\lp 0,q\rp\rp.
\end{align*}
Note then $R_c$ satisfies the Clifford condition
\[
    R_c\lp q\rp^2=-|q|^2
\]
and $R_c$ commutes with the natural right action of $\bH$ on $\bO$
\[
    \lp a,b\rp\cdot q:=\lp a\cdot q,b\cdot q\rp.
\]
Thus the image of $R_c$ lies within the smaller algebra of quaternionic linear maps on $\bO$
\[
    R_c:\bH\rightarrow \End_{\bH}\bO.
\]
The above maps induces an isomorphism of the Clifford algebra in dimension $4$
\[
    R_c:C\ell_r\xrightarrow{\sim}\End_{\bH}\bO.
\]
The careful reader may have noted that the above construction is really just a different way of book keeping the $C\ell_4$ representation on $\wedge_{\bH}\bH$, after all that is what it means to have isomorphic representations! The interesting thing is that the above trick of embedding $\bR^4$ into the imaginary part of $\bO$ extends through to $\bR^7$
\[
    c_k:\bR^k\rightarrow\text{Im}\bO,\quad 4\leq k\leq 7
\]
and gives an irreducible spinor module
\[
    C\ell_k\circlearrowright\bO,\quad 4\leq k\leq 7.
\]
One can see that if we want to extend this to $\bR^8$, we would need to embed into the imaginary part of some $\bR$-division algebra. One could perhaps turn to the \textbf{\emph{Sedonions}}, but at this point the non-associativity, as well as not possessing the property of being alternative encourages a different take. We can double the octonions $\bO\oplus\bO$ and define the map
\begin{align*}
    c_8:\bO&\rightarrow\End_{\bR}\bO\oplus\bO\\
    x&\mapsto \lp \lp u,v\rp\mapsto \lp -\overline{x}v,xu\rp\rp.
\end{align*}
Although the Octonions are not associative, they are alternative which implies that the above map satisfies
\[  
    c_8\lp x\rp^2=-|x|^2.
\]
Thus we have obtained an irreducible spinor representation
\[
    C\ell_8\circlearrowright \bO\oplus\bO.
\]
\section{The Case of General Signature Metrics}
We now take up the general task of giving a recipe for cooking up irreducible spinor modules for non-degenerate symmetric bilinear forms $g_{r,s}$, where $r$ is the number of $-1$'s in the matrix of $g_{r,s}$ relative to a $g_{r,s}$-orthonormal basis.\par
\subsection{The Signature $(n,0)$-case}
We first remark that the the Clifford algebras $C\ell_{n,0}$ also satisfy the $8$-fold Bott periodicity, however the algebras are distinct from their Euclidean cousins $C\ell_{0,n}$
\begin{align*}
    C\ell_{1,0}&\simeq\bR\oplus\bR,\quad
    C\ell_{2,0}\simeq M_2(\bR),\quad
    C\ell_{3,0}\simeq M_2(\bC),\quad
    C\ell_{4,0}\simeq M_2(\bH)\\
    C\ell_{5,0}\simeq &M_2(\bH)\oplus M_2(\bH),\quad
    C\ell_{6,0}\simeq M_4(\bH),\quad
    C\ell_{7,0}\simeq M_8(\bC),\quad
    C\ell_{8,0}\simeq M_{16}(\bR)
\end{align*}
and $C\ell_{n+8,0}\simeq C\ell_{n,0}\otimes_{\bR}C\ell_{8,0}$. We proceed as in the case of studying $C\ell_{0,n}$-spinors. Thus, suppose that $S_{i,0}$ are irreducible real representations of $C\ell_{i,0}$ for $1\leq i\leq 4$. As before, the Clifford algebra as well as the intertwiner algebras are symmetric and thus there are two versions of a given irreducible module
\[
    c_{i,0}^L:\bR^i\rightarrow\End^L_{\bK_{i,0}}\lp S_{i,0}\rp,\quad c^R_i:\bR^i\rightarrow\End^R_{\bK_{i,0}}\lp S_{i,0}\rp.
\]
We further remark that $\End_{C\ell_{i,0}}\lp S_{i,0}\rp\simeq \bR$ for $i=1,2$, $\End_{C\ell_{3,0}}\lp S_{3,0}\rp\simeq \bC$, and $\End_{C\ell_{4,0}}\lp S_{4,0}\rp\simeq\bH$. Lastly, just as in the $(n,0)$-theory, to obtain the other irreducible representation of $C\ell_{1,0}$, we compose $c_1$ with $-1$, thus obtaining $-S_{1,0}$ which is distinct as a module from $S_{1,0}$.\par
When $n=0\lp\text{ mod }4\rp$, a real spinor module $M_{n,0}$ is $\bH-\bZ_2$-graded by the volume element
\[
    \boldsymbol{\nu}:=e_1\cdots e_n, \quad M^{\pm}_{n,0}=\lp 1\pm\boldsymbol{\nu}\rp M_{n,0}.
\]
We use this $\bZ_2$-grading on $S_{4,0}$ to construct $S_{8,0}:=S_{4,0}\hat{\otimes}_{\bH}S_{4,0}$ and take
\begin{align*}
    c_{8k,0}:\bR^4\oplus\bR^4&\longrightarrow \End_{\bR}\lp S_{8,0}\rp\\
    \lp u,v\rp&\longmapsto c^R_{4,0}(u)\hat{\otimes}_{\bH}I_{S_{4,0}}+I_{S_{4,0}}\hat{\otimes}_{\bH}c^L_{4,0}(v).
\end{align*}
From here the construction of spinor modules in general $(n,0)$-dimensions is a direct translation of the Euclidean side. We write below the spinor modules and leave to the reader to check the details for themself.
\begin{align*}
    S_{8k,0}&:=\lp S_{8,0}\rp^{\hat{\otimes}^k_{\bR}},\\
    S_{8k+r,0}&:=S_{8k,0}\otimes_{\bR}S_{r,0},\quad 1\leq r\leq 3,\\
    S_{8k+4,0}&:=S_{8k,0}\hat{\otimes}_{\bR}S_{4,0},\\
    S_{8k+5,0}&:=S_{8k+4,0}\otimes_{\bR}S_{1,0},\\
    S_{8k+6,0}&:=S_{8k+4,0}\otimes_{\bR}S_{2,0},\\
    S_{8k+7,0}&:=S_{8k+4,0}\otimes_{\bC}S_{3,0}.
\end{align*}
Whenever there are factors of $S_{1,0}$, replacing by $-S_{1,0}$ gives a distinct irreducible module.
\subsection*{The General Case}
Suppose that $C\ell_{i,i}\circlearrowright S_{i,i}$ is a family of $\bR-\bZ_2$-graded irreducible $C\ell_{i,i}$ modules, $i\geq 1$. It is a theorem that in this case, $S_{i,i}$ is the unique up to isomorphism module for $C\ell_{i,i}$, and we will show below that this module is naturally $\bR-\bZ_2$-graded, with intertwiner algebra is isomorphic to $\bR$. Given a general $n=r+s$, we choose the maximal $i\geq 0$ such that $(r,s)-(i,i)$ lies within the first quadrant. In this case either $(r,s)-(i,i)=(a,0)$ or $(r,s)-(i,i)=(0,b)$. In the first case
\begin{align*}
    c_{r,s}:\bR^{2i}\oplus\bR^a&\longrightarrow \End^R_{\bK_{(a,0)}}\lp S_{i,i}\otimes_{\bR} S_{a,0}\rp\\
    (u,v)&\longmapsto c_{i,i}(u)\otimes_{\bR}I_{S_{a,0}}+I_{S_{i,i}}\hat{\otimes}_{\bR}c^R_{a,0}(v)
\end{align*}
gives an irreducible representation of $C\ell_{r,s}$.
In the second case
\begin{align*}
    c_{r,s}:\bR^{2i}\oplus\bR^b&\longrightarrow \End^R_{\bK_{(0,b)}}\lp S_{i,i}\otimes_{\bR} S_{0,b}\rp\\
    (u,v)&\longmapsto c_{i,i}(u)\otimes_{\bR}I_{S_{0,b}}+I_{S_{i,i}}\hat{\otimes}_{\bR}c^R_{0,b}(v)
\end{align*}
gives an irreducible representation of $C\ell_{r,s}$. As before, anytime a $S_{1,0}$ or $S_{0,3}$ factor appear within the tensor product, a simple swap to the negative representation gives the other distinct irreducible representation of $C\ell_{r,s}$. This completes the general recipe for cooking up all possible modules 
\section{A Choice of Irreducible Family}
Staying in line with the quaternic multi-vector approach we developed for the $(n,0)$ representation theory, we will construct a family of irreducible representations based on $\bR$, $\bC$ and $\bH$ multi-vectors. 
\subsection{The Signature $(n,0)$-case}
In dimension $(1,0)$, the linear map
\begin{align*}
    c_{1,0}:\bR&\longrightarrow \End_{\bR}\lp \bR\rp\\
    x&\longmapsto \lp y\mapsto x\cdot y\rp
\end{align*}
gives a $C\ell_{1,0}\circlearrowright \bR$ irreducible action.\par
In dimension $(2,0)$, if $e_1,e_2$ denote a $g_{2,0}$ orthonormal basis of $\bR^2$, and $f_1$ a spanning vector for $\bR$, then the operators on $\wedge_{\bR}\bR$
\begin{align*}
    c_{2,0}(e_1)&=f_1\wedge_{\bR}+\iota_{f_1}\\
    c_{1,0}(e_2)&=\epsilon
\end{align*}
where $\epsilon$ denotes the grading-by-degree operator on $\wedge_{\bR}\bR$. The $\bR$-linear extension of $c_{2,0}$ to $\bR^2$ gives the irreducible representation $C\ell_{2,0}\circlearrowright \wedge_{\bR}\bR$.\par
In dimension $(3,0)$, if $e_1,e_2,\text{ and }e_3$ denote a $g_{3,0}$ orthonormal basis of $\bR^3$, and if $f_1$ denotes a spanning vector for $\bC$ with Hermitian dual $\overline{f}_1\in\bC^*$, then the complex linear operators on $\wedge_{\bC}\bC$ given by
\begin{align*}
    c_{3,0}(e_1)&=f_1\wedge_{\bC}+\iota_{\overline{f}_1}\\
    c_{3,0}(e_2)&=if_1\wedge_{\bC}-i\iota_{\overline{f}_1}\\
    c_{3,0}(e_3)&=\epsilon
\end{align*}
when $\bR$-linearly extended give $C\ell_{3,0}\circlearrowright\wedge_{\bC}\bC$ irreducible action.\par
In dimension $(4,0)$, we use our $(0,4)$ irreducible quaternic module $\wedge_{\bH}\bH$, and treating $\bR^4=\bH$, we take as generating linear map
\begin{align*}
    c_{4,0}:\bH&\longmapsto\End^R_{\bH}\lp\wedge_{\bH}\bH\rp\\
    q&\longmapsto \varepsilon^L_q+\iota^L_q
\end{align*}
which gives the irreducible action $C\ell_{4,0}\circlearrowright\wedge_{\bH}\bH$.
\subsection{The Signature $(i,i)$-case}
Let $\lp\bR^i\rp^*$ denote the $\bR$-dual space of $\bR^i$. Then the vector space $\bR^i\oplus\lp\bR^i\rp^*$ has a natural metric given by
\[
    g\lp \lp x,\omega\rp,\lp y,\tau\rp\rp:=\omega(y)+\tau(x),\quad \lp x,\omega\rp,\lp y,\tau\rp\in\bR^i\oplus\lp\bR^i\rp^*
\]
which is signature $(i,i)$. The linear map
\begin{align*}
    c_{i,i}:\bR^i\oplus\lp\bR^i\rp^*&\longrightarrow\End_{\bR}\lp\wedge_{\bR}\bR^i\rp\\
    \lp x,\omega\rp&\longmapsto x\wedge_{\bR}-\iota_{\omega}
\end{align*}
generates the irreducible $C\ell_{i,i}\circlearrowright\wedge_{\bR}\bR^i$. Moreover, the representation is clearly $\bZ_2$-graded by using the even/odd grading on multivectors. An application of these spinors modules can be seen in \cite{Per} where an algebraic proof of the Atiyah-Singer index theorem is presented and generalized to the case of foliated Heisenberg manifolds in \cite{PR}.

\end{document}